\documentclass{article}
\usepackage{amssymb,amsfonts,amsmath,amsthm,amsopn,mathrsfs,amstext,amscd,latexsym,xy,color,stmaryrd,epic}
\usepackage{hyperref}
\usepackage{csquotes}
\usepackage{tikz,multicol,graphicx,longtable}
\usetikzlibrary{arrows,positioning,decorations.pathmorphing,
  decorations.markings
 }
\tikzset{inner sep=0pt, 
  root/.style={circle,draw,minimum size=7pt,thick}, 
  fatroot/.style={circle,draw,minimum size=10pt,thick}, 
  short root/.style={circle,fill,minimum size=7pt}, 
  doublearrow/.style={postaction={decorate}, 
  decoration={markings,mark=at position .7
  with {\arrow{angle 60}}},double distance=3pt,thick}
} 

\theoremstyle{plain}
\input xy
\xyoption{all}
\newtheorem{theorem}{Theorem}[section]
\newtheorem{lemma}[theorem]{Lemma}
\newtheorem{proposition}[theorem]{Proposition}

\newtheorem{definition}[theorem]{Definition}
\newtheorem{conjecture}[theorem]{Conjecture}

\newtheorem{principle}[theorem]{Principle}
\theoremstyle{remark}

\numberwithin{equation}{section}
\numberwithin{paragraph}{section}

\DeclareMathOperator{\Hom}{Hom}

\DeclareMathOperator{\Ad}{Ad}
\DeclareMathOperator{\Cent}{Cent}
\DeclareMathOperator{\rank}{rank}

\DeclareMathOperator{\Gal}{Gal}
\DeclareMathOperator{\Aut}{Aut}

\DeclareMathOperator{\Res}{Res}
\DeclareMathOperator{\Sp}{Sp}

\DeclareMathOperator{\Isom}{Isom}
\DeclareMathOperator{\Ind}{Ind}

\DeclareMathOperator{\tr}{tr}

\DeclareMathOperator{\diag}{diag}

\DeclareMathOperator{\Frob}{Frob}
\DeclareMathOperator{\Sh}{Sh}
\DeclareMathOperator{\Spec}{Spec}
\DeclareMathOperator{\Mot}{Mot}

\newcommand{\cF}{{\mathcal F}}

\newcommand{\cH}{{\mathcal H}}

\newcommand{\cO}{{\mathcal O}}

\newcommand{\cV}{{\mathcal V}}

\newcommand{\fra}{{\mathfrak a}}

\newcommand{\frg}{{\mathfrak g}}
\newcommand{\frh}{{\mathfrak h}}

\newcommand{\frl}{{\mathfrak l}}
\newcommand{\ffrm}{{\mathfrak m}}
\newcommand{\frn}{{\mathfrak n}}

\newcommand{\frs}{{\mathfrak s}}

\newcommand{\frt}{{\mathfrak t}}

\newcommand{\frz}{{\mathfrak z}}
\newcommand{\bbA}{{\mathbb A}}

\newcommand{\bbC}{{\mathbb C}}

\newcommand{\bbG}{{\mathbb G}}

\newcommand{\bbQ}{{\mathbb Q}}
\newcommand{\bbR}{{\mathbb R}}

\newcommand{\bbZ}{{\mathbb Z}}

\newcommand{\GL}{\mathrm{GL}}

\newcommand{\PGL}{\mathrm{PGL}}

\newcommand{\SL}{\mathrm{SL}}

\newcommand{\ol}{\overline}
\newcommand{\ul}{\underline}
\newcommand{\sub}{\subseteq}
\newcommand{\ms}{\mathscr}

\title{On subquotients of the \'etale cohomology of Shimura varieties}
\author{Christian Johansson and Jack A. Thorne}

\begin{document}
\maketitle

\section{Introduction}\label{sec_introduction}

Let $L$ be a number field, and let $\pi$ be a cuspidal automorphic representation of $\GL_n(\bbA_L)$. Suppose that $\pi$ is $L$-algebraic and regular. By definition, this means that for each place $v | \infty$ of $L$, the Langlands parameter
\[ \phi_v : W_{L_v} \to \GL_n(\bbC) \]
of $\pi_v$ has the property that, up to conjugation, $\phi_v|_{\bbC^\times}$ is of the form $z \mapsto z^\lambda \overline{z}^\mu$ for regular cocharacters $\lambda, \mu$ of the diagonal torus of $\GL_n$. In this case we can make, following \cite{Clo90} and \cite{Buz14}, the following conjecture:
\begin{conjecture}\label{conj_galois}
For any prime $p$ and any isomorphism $\iota : \overline{\bbQ}_p \cong \bbC$, there exists a continuous, semisimple representation $r_{p,\iota}(\pi) : \Gamma_L \to \GL_n(\overline{\bbQ}_p)$ satisfying the following property: for all but finitely many finite places $v$ of $L$ such that $\pi_v$ is unramified, $r_{p, \iota}(\pi)|_{\Gamma_{L_v}}$ is unramified and the semisimple conjugacy class of $r_{p, \iota}(\pi)(\Frob_v)$ is equal to the Satake parameter of $\iota^{-1} \pi_v$.
\end{conjecture}
(We note that this condition characterizes $r_{p, \iota}(\pi)$ uniquely (up to isomorphism) if it exists, by the Chebotarev density theorem.) The condition that $\pi$ is $L$-algebraic and regular implies that the Hecke eigenvalues of a twist of $\pi$ appear in the cohomology of the arithmetic locally symmetric spaces attached to the group $\GL_{n, L}$. The first cases of Conjecture \ref{conj_galois} to be proved were in the case $n = 2$ and $L = \bbQ$, in which case these arithmetic locally symmetric spaces arise as complex points of Shimura varieties (in fact, modular curves), and the representations $r_{p, \iota}(\pi)$ can be constructed directly as subquotients of the $p$-adic \'etale cohomology groups (see e.g.\ \cite{Del71}). Similar techniques work in the case where $n = 2$ and $L$ is a totally real field (see e.g.\ \cite{Car86}).

The next cases of the conjecture to be established focused on the case where $L$ is totally real or CM and $\pi$ satisfies some kind of self-duality condition. When $n > 2$, or when $L$ is not totally real, the arithmetic locally symmetric spaces attached to the group $\GL_{n, L}$ do not arise from Shimura varieties. However, the self-duality condition implies that $\pi$ or one of its twists can be shown to descend to another reductive group $G$ which does admit a Shimura variety. In this case the representations $r_{p, \iota}(\pi)$ can often be shown to occur as subquotients of the $p$-adic \'etale cohomology groups of the Shimura variety associated to some Shimura datum $(G, X)$. The prototypical case is when $L$ is a CM field and there is an isomorphism $\pi^c \cong \pi^\vee$, where $c \in \Aut(L)$ is complex conjugation. In this case $\pi$ descends to a cuspidal automorphic representation $\Pi$ of a unitary (or unitary similitude) group $G$ such that $\Pi_\infty$ is essentially square-integrable.

Going beyond the case where $\pi$ satisfies a self-duality condition requires new ideas. The general case of Conjecture \ref{conj_galois} where $L$ is a totally real or CM field was established in \cite{Har16} (another proof was given shortly afterwards in \cite{Sch16}). The difficulty in generalizing the above techniques to the case where $\pi$ is not self-dual is summarized in \cite{Har16} as follows: 
\begin{equation}\label{intro_quote}
\begin{split}&\text{According to unpublished computations of one of us (M.H.) and of Laurent Clozel, in the non-}\\&\text{polarizable case the representation } r_{p, \iota}(\pi) \text{ will never occur in the cohomology of a Shimura variety.}
\end{split}
\end{equation}
The purpose of this note is to expand on the meaning of this statement. According to \cite{Bar14}, an irreducible Galois representation is polarizable if it is conjugate self-dual up to twist. We first prove a negative result, showing that there are many Galois representations which are not conjugate self-dual up to twist and which do appear in the cohomology of Shimura varieties:
\begin{theorem}\label{thm_intro_automorphy}
Let $p$ be a prime, and fix an isomorphism $\iota : \overline{\bbQ}_p \cong \bbC$. Then there exist infinitely pairs $(L, \pi)$ satisfying the following conditions:
\begin{enumerate}
\item $L \subset \bbC$ is a CM number field and $\pi$ is a regular $L$-algebraic cuspidal automorphic representation of $\GL_n(\bbA_L)$ such that $\pi^c \not\cong \pi^\vee \otimes \chi$ for any character $\chi : L^\times \backslash \bbA_L^\times \to \bbC^\times$.
\item There exists a Shimura datum $(G, X)$ of reflex field $L$ such that the associated Shimura varieties $\Sh_K(G, X)$ are proper and $r_{p, \iota}(\pi)$ appears as a subquotient of $H^\ast_{\text{\'et}}(\Sh_K(G, X)_{\overline{\bbQ}}, \cF_{\tau,p})$ for some algebraic local system $\cF_\tau$ and for some neat open compact subgroup $K \subset G(\bbA_\bbQ^\infty)$.
\end{enumerate}
\end{theorem}
(Here $\cF_{\tau, p}$ denotes the lisse $\overline{\bbQ}_p$-sheaf on $\Sh_K(G, X)$ associated to $\cF_\tau$; see \S \ref{sec_notation} below for the precise notation that we use.) It is therefore necessary to give a different interpretation to the assertion (\ref{intro_quote}). The representations $r_{p, \iota}(\pi)$ appearing in Theorem \ref{thm_intro_automorphy} are necessarily special: in fact, the examples we construct are induced from cyclic CM extensions of $L$. 

One subtlety here is that even if an irreducible representation $r : \Gamma_L \to \GL_n(\overline{\bbQ}_p)$ is conjugate self-dual up to twist (as one would expect e.g.\ for the $n$-dimensional representation attached to a RLACSDC\footnote{Regular $L$-algebraic, conjugate self-dual, cuspidal} automorphic representation of $\GL_n(\bbA_L)$), it need not be the case that the irreducible subquotients of tensor products $r^{\otimes a} \otimes (r^\vee)^{\otimes b}$ are conjugate self-dual up to twist (and indeed, it is this possibility that we exploit in our proof of Theorem \ref{thm_intro_automorphy}). This points to the need to phrase a condition in terms of the geometric monodromy group of $r$ (i.e.\ the identity component of the Zariski closure of $r(\Gamma_L)$). The Galois representations that we construct in the proof of Theorem \ref{thm_intro_automorphy} are at least `geometrically polarizable', in the sense that complex conjugation induces the duality involution on the geometric monodromy group. The main point we make in this paper is that well-known conjectures imply that all Galois representations appearing in the cohomology of Shimura varieties are geometrically polarizable, using statements like our Principle \ref{principle_oddness} below. (In the body of the paper, we use the terminology `odd' instead of `geometrically polarizable'; see Definition \ref{def_odd_galois_representation}.)

In order to fully address the question posed in (\ref{intro_quote}), one must first answer the question of which kind of cohomology groups to consider. If $\Sh_K(G, X)$ is proper then ordinary \'etale cohomology with coefficients in an algebraic local system provides the only natural choice. In the non-compact case, one could consider ordinary cohomology, cohomology with compact support, or the intersection cohomology of the minimal compactification $\Sh_K^\text{min}(G, X)$ of $\Sh_K(G, X)$. We first study the intersection cohomology, using its relation with discrete automorphic representations of $G(\bbA_\bbQ)$. This leads, for example, to the following theorem.
\begin{theorem}\label{thm_intro_non_automorphy}
Let $L$ be an imaginary CM or totally real number field, and let $\rho : \Gamma_L \to \GL_n(\overline{\bbQ}_p)$ be a continuous representation which is strongly irreducible, in the sense that for any finite extension $M / L$, $\rho|_{\Gamma_M}$ is irreducible. Let $(G, X)$ be a Shimura datum of reflex field $L$. Assume Conjecture \ref{conj_galois_reps_for_non_tempered_representations} and Conjecture \ref{conj_consequence_of_Arthur_Kottwitz}. Let $j : \Sh_{K}(G, X) \to \Sh^\text{min}_K(G, X)$ be the open immersion of the Shimura variety into its minimal compactification. If $\rho$ appears as a subquotient of $H^\ast_{\text{\'et}}(\Sh^\text{min}_K(G, X)_{\overline{\bbQ}}, j_{! \ast} \cF_{\tau, p})$ for some algebraic local system $\cF_\tau$, then $\rho$ is conjugate self-dual up to twist. 
\end{theorem}
It is easy to construct examples of strongly irreducible Galois representations which are not conjugate self-dual up to twist (for example, arising from elliptic curves over an imaginary CM field $L$). Conjecturally, then, these Galois representations can never appear as subquotients of the intersection cohomology of Shimura varieties. 

We can summarise the conjectures assumed in the statement of Theorem \ref{thm_intro_non_automorphy} as follows. Conjecture \ref{conj_galois_reps_for_non_tempered_representations} asserts the existence of Galois representations attached to discrete cohomological automorphic representations $\pi$ of $G(\bbA_\bbQ)$, where $G$ is a reductive group over $\bbQ$ such that $G(\bbR)$ admits discrete series. It includes a rather precise formulation of local-global compatibility at infinity based on a connection with an $A$-parameter of $\pi_\infty$. (For a closely related statement, see \cite[\S 8.2.3.4]{Ser12}.) It would not be possible to formulate this using only the formalism of $L$-parameters (as opposed to $A$-parameters). Conjecture \ref{conj_consequence_of_Arthur_Kottwitz} is a weak consequence of Kottwitz's conjectural description of the intersection cohomology of the minimal compactification of a Shimura variety in terms of $A$-parameters, slightly reformulated here in a similar manner to \cite{Joh13}. This focus on $A$-parameters is essential, since for a result like Theorem \ref{thm_intro_non_automorphy} the most interesting part of cohomology is indeed the part corresponding to non-tempered automorphic representations. 

Since compactly supported cohomology is dual to ordinary cohomology, the other case to consider is that of the ordinary cohomology of non-proper Shimura varieties. In this case, Morel's theory of weight truncations can be used to reduce to the case of intersection cohomology. This leads, for example, to the following theorem.
\begin{theorem}\label{thm_consequence_of_morel}
Let $(G, X)$ be a Shimura datum satisfying the assumptions of \S \ref{open}, and let $L$ be its reflex field. Let $\cF_\tau$ be an algebraic local system on $\Sh_K(G, X)$, and let $p$ be a prime. Then any irreducible subquotient $\overline{\bbQ}_p[\Gamma_L]$-module of $H^\ast_{\text{\'et}}(\Sh_K(G, X)_{\overline{\bbQ}}, \cF_{\tau, p})$ is isomorphic to a subquotient of $H^\ast(\Sh^\text{min}_{K'}(G', X')_{\overline{\bbQ}}, j_{! \ast} \cF_{\tau', p})$ for some Shimura datum $(G', X')$ of reflex field $L$. 
\end{theorem}
We note that the assumptions in \S \ref{open} hold in particular for the Shimura data associated to inner forms of unitary similitude and symplectic similitude groups. These are the groups used in \cite{Har16} and \cite{Sch16}, and which led us to be interested in these problems in the first place.

We now describe the organization of this note in more detail. In \S \ref{sec_galois_representations} we review some principles from the representation theory of reductive groups, and consequences for what we call `odd Galois representations'. In \S \ref{sec_positive_results} we prove Theorem \ref{thm_intro_automorphy} by explicitly constructing irreducible Galois representations in the cohomology of unitary Shimura varieties which are not conjugate self-dual up to twist. In \S \ref{sec_conjectures_on_galois_representations}, we introduce the Langlands group and the formalism of $L$-parameters and $A$-parameters, and use this as a heuristic in order to justify Conjecture \ref{conj_galois_reps_for_non_tempered_representations}. In \S \ref{sec_negative_results} we combine this with Kottwitz's conjectural description of the intersection cohomology of Shimura varieties in order to state Conjecture \ref{conj_consequence_of_Arthur_Kottwitz} and then, using the groundwork done in \S \ref{sec_galois_representations}, to prove Theorem \ref{thm_intro_non_automorphy}. Finally, in \S \ref{open} we sketch Morel's proof of  Theorem \ref{thm_consequence_of_morel}.

\subsection{Acknowledgments}

We are very grateful to Laurent Clozel, Michael Harris, Frank Calegari, and Toby Gee for their comments on an earlier draft of this paper. This work was begun while JT served as a Clay Research Fellow. This project has received funding
from the European Research Council (ERC) under the European Union's Horizon
2020 research and innovation programme (grant agreement No 714405).

\subsection{Notation}\label{sec_notation}

A reductive group is not necessarily connected. If $G, H, \dots$ are linear algebraic groups over a field $\Omega$ of characteristic 0, then use gothic letters $\frg, \frh, \dots$ to denote their respective Lie algebras. We write $\frs\frl_2$ for the Lie algebra of $\SL_2$; it has a basis of elements
\[ x_0 = \left( \begin{array}{cc} 0 & 1 \\ 0 & 0 \end{array} \right),  t_0 = \left( \begin{array}{cc} 1 & 0 \\ 0 & -1 \end{array} \right),  y_0 = \left( \begin{array}{cc} 0 & 0 \\ 1 & 0 \end{array} \right). \]
These satisfy the relations 
\[ [ x_0, y_0 ] = t_0, [ t_0, x_0] = 2 x_0, [t_0, y_0] = - 2 y_0. \]
If $\frg$ is any Lie algebra and $(x, t, y)$ is a tuple of elements of $\frg$ satisfying the same relations, then we call $(x, t, y)$ an $\frs\frl_2$-triple in $\frg$.

If $E$ is a field, then we write $\Gamma_E$ for the absolute Galois group of $E$ with respect to some fixed separable closure $\overline{E}$.  If $E$ is a number field and $v$ is a place of $E$, then we write $\Gamma_{E_v} \subset \Gamma_E$ for the decomposition group at $v$, which is well-defined up to conjugation. If $v$ is a non-archimedean place, then we also write $k(v)$ for the residue field of $v$, and $q_v$ for the cardinality of $k(v)$. We writ $\Frob_v$ for a geometric Frobenius element. We write $\bbA_E$ for the ad\`ele ring of $E$, and $\bbA_E^\infty$ for its finite part. If $p$ is a fixed prime, then we write $\epsilon : \Gamma_E \to \bbQ_p^\times$ for the $p$-adic cyclotomic character.

If $G$ is a connected reductive group over $\bbQ$, then we write ${}^L G$ for its $L$-group, which we usually think of as a semi-direct product $\widehat{G} \rtimes \Gal(E / \bbQ)$, where $\widehat{G}$ is the dual group (viewed as a split reductive group over $\bbQ$) and $E / \bbQ$ is the Galois extension over which $G$ becomes an inner form of its split form. If $p$ is a prime, then an $L$-homomorphism $\rho : \Gamma_\bbQ \to {}^L G(\overline{\bbQ}_p)$ is a homomorphism for which the projection $\Gamma_\bbQ \to \Gal(E / \bbQ)$ is the canonical one.

If $(G, X)$ is a Shimura datum, in the sense of \cite{Del79}, and $K \subset G(\bbA_\bbQ^\infty)$ is an open compact subgroup, then we write $\Sh_K(G, X)$ for the associated Shimura variety, which is an algebraic variety defined over the reflex field of the pair $(G, X)$ (see \cite{Mil83} for existence in the most general case). By an algebraic local system $\cF_\tau$, we mean the local system of $\overline{\bbQ}$-vector spaces $\cF_\tau$ on $\Sh_K(G, X)(\bbC)$ associated to a finite-dimensional algebraic representation $\tau : G_{\overline{\bbQ}} \to \GL(V_\tau)$ such that the central character $\omega_\tau : Z(G)_{\overline{\bbQ}} \to \bbG_m$ is defined over $\bbQ$. If $p$ is a prime and $\iota : \overline{\bbQ} \hookrightarrow \overline{\bbQ}_p$ is a fixed embedding, then we get a lisse \'etale sheaf $\cF_{\tau, p}$ on $\Sh_K(G, X)$, which is the one considered in the introduction to this paper. 

\section{Odd Galois representations}\label{sec_galois_representations}

In this section we discuss Galois representations $\rho : \Gamma_\bbQ \to H(\overline{\bbQ}_p)$, where $H$ is a reductive group. We are particularly interested in representations which are odd, in the sense of Definition \ref{def_odd_galois_representation}. If $H = {}^L G$, where $G$ is a connected reductive group over $\bbQ$ such that $G^\text{ad}(\bbR)$ contains a compact maximal torus, then Definition \ref{def_odd_galois_representation} coincides with the one given in F. Calegari's note \cite{Cal}, but not otherwise (see also \cite{Gro}).

Let $\Omega$ be a field of characteristic 0.
\begin{definition}\label{def_odd_involution}
Let $G$ be a connected reductive group over $\Omega$. We say that an involution $\theta : G \to G$ is odd if $\tr( d \theta^\text{ad} : \frg^\text{ad} \to \frg^\text{ad}) = - \rank G^\text{ad}$, where $\theta^\text{ad}$ is the induced involution of the adjoint group $G^\text{ad}$.
\end{definition}
If $G$ is semisimple, then the class of odd involutions coincides with the class of Chevalley involutions (see e.g.\ \cite{Ada16}). In particular, they are all $G(\overline{\Omega})$-conjugate. However, in general we diverge from this class by allowing also involutions which e.g.\ act trivially on the centre of $G$.
\begin{lemma}\label{lem_oddness_in_reductive_subgroups}
Let $G$ be a reductive group over $\Omega$, and let $H \subset G$ be a closed reductive subgroup. Let $\theta$ be an involution of $G$ which leaves $H$ invariant and such that $\theta|_{G^\circ}$ is odd. Suppose that there exist cocharacters $\mu, w : \bbG_m \to H$ satisfying the following properties:
\begin{enumerate}
\item $\mu$ is regular in $G^\circ$, and $w$ takes values in $Z(G)$.
\item $\theta \circ \mu = \mu^{-1} w$.
\end{enumerate}
Then $\theta|_{H^\circ}$ is an odd involution.
\end{lemma}
\begin{proof}
After replacing $G$ by its quotient by $Z(G^\circ)$, we can assume that $G^\circ$ is adjoint and $w = 1$. Let $T = Z_G(\mu)$, and consider the decomposition 
\[ \frg = \frt \oplus \frg^+ \oplus \frg^- \] into zero, positive, and negative weight spaces for the cocharacter $\mu$. Since $\theta \circ \mu = \mu^{-1}$, we see that $\theta$ swaps $\frg^+$ and $\frg^-$. Since $\tr d \theta = - \dim \frt$, we see that $\theta$ must act as $-1$ on $\frt$. Since $\mu$ factors through $H$, we find a similar decomposition
\[ \frh = (\frh \cap \frt) \oplus \frh^+ \oplus \frh^-, \]
showing that $\tr d \theta|_\frh = - \rank H^\circ$, and hence that $\theta|_{H^\circ}$ is also odd. 
\end{proof}
\begin{definition}\label{def_odd_galois_representation}
Let $p$ be a prime and let $G$ be a reductive group over $\overline{\bbQ}_p$, and let $\rho : \Gamma_\bbQ \to G(\overline{\bbQ}_p)$ be a continuous representation. We say that $\rho$ is odd if $\theta = \Ad \rho(c)|_{G^\circ}$ is an odd involution.
\end{definition}
Now let $E$ be a number field, and let $G$ be a reductive group over $\overline{\bbQ}_p$ for some prime $p$.
\begin{definition}\label{def_mixed}
Let $\rho : \Gamma_E \to G(\overline{\bbQ}_p)$ be a continuous representation which is unramified almost everywhere. Then:
\begin{enumerate}
\item We say that $\rho$ is mixed if there exists a cocharacter $w : \bbG_m \to G$ centralizing the image of $\rho$ and such that for any representation $G \to \GL(V)$, $V \circ \rho$ is mixed with integer weights, and $V = \oplus_{i \in \bbZ} V^{w(t) = t^i}$ is its weight decomposition. In other words, there exists a finite set of finite places of $E$, containing the $p$-adic places, such that for any finite place $v \not\in S$ of $E$, $\rho|_{\Gamma_{E_v}}$ is unramified and for any isomorphism $\iota : \overline{\bbQ}_p \cong \bbC$, any eigenvalue $\alpha$ of $\Frob_v$ on $V^{w(t) = t^i}$ satisfies $\iota(\alpha) \overline{\iota(\alpha)} = q_v^i$.
\item We say that $\rho$ is pure if it is mixed and $w$ takes values in $Z(G)$.
\item We say that $\rho$ is geometric if for each place $v | p$ of $E$, $\rho|_{\Gamma_{E_v}}$ is de Rham. In other words, for any representation $G \to \GL(V)$, $V \circ \rho|_{\Gamma_{E_v}}$ is de Rham in the sense of $p$-adic Hodge theory.
\end{enumerate}
\end{definition}
Note that if $\rho$ is mixed, then $w$ is uniquely determined by $\rho$. When working with a mixed Galois representation, we will always write $w$ for its corresponding weight cocharacter.

Let $\bbC_p$ denote the completion of $\overline{\bbQ}_p$. If $\rho$ is de Rham then it is also Hodge--Tate, so there exists a cocharacter $\mu_{\text{HT}} : \bbG_m \to G_{\bbC_p}$, again uniquely determined, such that $V_{\bbC_p} = \oplus_{i \in \bbZ} V_{\bbC_p}^{\mu_{\text{HT}}(t) = t^i}$ is the Hodge--Tate decomposition of $V_{\bbC_p}$. 
\begin{definition}
Let $\rho : \Gamma_E \to G(\overline{\bbQ}_p)$ be a geometric representation and let $v | p$ be a place of $E$. We call a Hodge--Tate cocharacter at $v$ any cocharacter $\mu : \bbG_m \to G$ with the following properties:
\begin{enumerate}
\item $\mu$ takes values in the Zariski closure $H$ of $\rho(\Gamma_E)$.
\item $\mu$ is $H(\bbC_p)$-conjugate to $\mu_{\text{HT}}$.
\end{enumerate}
\end{definition}
Note that Hodge--Tate cocharacters always exist. 

We conclude this section with a discussion of Hodge--Tate cocharacters satisfying special properties. This will be used as motivation in \S \ref{sec_negative_results}. Let $G$ be a reductive group over $\overline{\bbQ}_p$, and let $\rho : \Gamma_\bbQ \to G(\overline{\bbQ}_p)$ be a geometric representation of Zariski dense image. In this case, we expect that the following should be true:
\begin{itemize}
\item $\rho$ is pure. Let $w : \bbG_m \to Z(G)$ denote the corresponding character.
\item There exists a Hodge--Tate cocharacter $\mu : \bbG_m \to G$ and a complex conjugation $c \in \Gamma_\bbQ$ such that $\Ad(\rho(c))(\mu) = \mu^{-1} w$.
\end{itemize}
Indeed, let us suppose that we are in the ``paradis motivique'' described in \cite{Ser94} (in other words, we assume the standard conjectures and the Tate conjecture). We are free to replace  $\rho$ by $\rho \times \epsilon$ and $G$ by the Zariski closure of the image of Galois in $G \times \bbG_m$. According to the conjectures in \cite{Fon95}, we should be able to find a faithful representation $R : G \to \GL(V)$ such that $R \circ \rho$ appears as a subquotient of the \'etale cohomology of a smooth projective variety $X$ over $\bbQ$. Let $G_X \subset \GL(H^\ast(X_{\overline{\bbQ}}, \bbQ_p))$ denote the Zariski closure of the image of $\Gamma_\bbQ$; then $G$ is isomorphic to a quotient of $G_{X, \overline{\bbQ}_p}$, so we just need to justify the existence of a Hodge--Tate cocharacter $\mu : \bbG_m \to G_{X, \overline{\bbQ}_p}$ and complex conjugation $c \in \Gamma_\bbQ$ satisfying the expected properties. 

We will use the language of of Tannakian categories, as in \cite{Ser94}. Let $\Mot_\bbQ$ denote the Tannakian category of motives over $\bbQ$, and let $\langle X \rangle$ denote its tensor subcategory generated by $X$. Let $\mathrm{Vec}_{\bbQ}$ denote the tensor category of finite-dimensional $\bbQ$-vector spaces. Then there are Hodge and Betti fibre functors
\[ \omega_H : \langle X \rangle \to \mathrm{Vec}_\bbQ, X \mapsto \oplus_{i, j} H^i(X, \Omega_X^j) \]
and
\[ \omega_B : \langle X \rangle \to \mathrm{Vec}_\bbQ, X \mapsto H^\ast(X(\bbC), \bbQ). \]
There is a Hodge--Betti comparison isomorphism $\alpha \in \Isom^\otimes(\omega_{H}, \omega_B)(\bbC)$. If we fix a choice of isomorphism $\bbZ_p \cong \bbZ_p(1)$ of $\bbZ_p$-modules and an embedding $\overline{\bbQ} \hookrightarrow \bbC$, then there is determined an isomorphism $\beta \in \Isom^\otimes(\omega_H, \omega_B)(\bbC_p)$ (cf. \cite{Fal88}). We write $M_{X, B} = \Aut^\otimes(\omega_B)$ for the usual motivic Galois group and $c \in M_{X, B}(\bbQ)$ for the image of complex conjugation, and $M_{X, H} = \Aut^\otimes(\omega_H)$. Then the Hodge grading determines a Hodge cocharacter $\mu_H : \bbG_m \to M_{X, H}$ which satisfies
\[ \Ad(c) \circ ( \alpha \circ \mu_H \circ \alpha^{-1}) = w \Ad(c) \circ ( \alpha \circ \mu_H^{-1} \circ \alpha^{-1}), \]
where $w$ is the weight cocharacter, which is central and defined over $\bbQ$ (cf. \cite[\S 0.2.5]{Del79a}). Fix an isomorphism $\iota : \overline{\bbQ}_p \to \bbC$. We'll be done if we can show that $\iota^{-1} (\alpha \circ \mu_H \circ \alpha^{-1}) = (\iota^{-1} \alpha) \circ \mu_H \circ (\iota^{-1} \alpha)^{-1}$ is a Hodge--Tate cocharacter, when we identify $M_{X, B, \bbQ_p}$ with the group $G_X$ above.

By definition, this means we must show that $(\iota^{-1} \alpha) \circ \mu_H \circ (\iota^{-1} \alpha)^{-1}$ is $M_{X, B}(\bbC_p)$-conjugate to the character $\beta \circ \mu_H \circ \beta^{-1}$. However, we have $\iota^{-1} \alpha \circ \beta^{-1} \in \Isom^\otimes(\omega_B, \omega_B)(\bbC_p) = M_{X, B}(\bbC_p)$, so this is automatic. 

Taking on board Lemma \ref{lem_oddness_in_reductive_subgroups} and the above `motivic' discussion, we arrive at the following unproven principle:
\begin{principle}\label{principle_oddness}
Let $G$ be a reductive group over $\overline{\bbQ}_p$, and let $i : H \to G$ be the embedding of a closed reductive subgroup. Let $\rho : \Gamma_\bbQ \to H(\overline{\bbQ}_p)$ be a geometric representation. Suppose the following:
\begin{enumerate}
\item $i \circ \rho$ is pure.
\item $i \circ \rho$ is odd.
\item The Hodge--Tate cocharacter of $i \circ \rho$ is regular in $G^\circ$.
\end{enumerate}
Then $\rho$ is odd. 
\end{principle}
It is instructive to discuss all of the above in a concrete example. Let us take the representation $\rho : \Gamma_\bbQ \to \GL_3(\overline{\bbQ}_p)$ constructed in \cite{Gee94}, and associated to a non-self dual cuspidal automorphic representation of $\GL_3(\bbA_\bbQ)$ of level $\Gamma_0(128)$. More precisely, we consider the representation constructed there inside the \'etale cohomology of a surface; the computations of Frobenius traces in \emph{op. cit.} support the hypothesis, but do not prove, that these representations are the same as the ones attached to the above-mentioned cuspidal automorphic representation.

The representation $\rho$ is irreducible, by the argument on \cite[p. 400]{Gee94}. In fact, $\rho$ has Zariski dense image in $\GL_3$. (This can be established using some $p$-adic Hodge theory. Write $H$ for the Zariski closure of $\rho(\Gamma_\bbQ)$. Then $H^\circ$ contains the image of a Hodge--Tate cocharacter, which is regular in $\GL_3$. If $H^\circ$ is abelian then there is an isomorphism $\rho \cong \Ind_{\Gamma_L}^{\Gamma_\bbQ} \chi$ for some degree 3 extension $L / \bbQ$ and geometric character $\chi : \Gamma_L \to \overline{\bbQ}_p^\times$. However, the infinity type of $\chi$ must be induced from the maximal CM subfield of $L$, which is totally real. This contradicts the fact that $\rho$ is Hodge--Tate regular. Therefore $H^\circ$ is not abelian. If the derived group of $H^\circ$ has rank 1, then it is equal to $\PGL_2$ in its 3-dimensional representation. The normalizer of $\PGL_2$ in $\GL_3$ is $\mathrm{GO}_3$. Since $\rho$ is not self-dual up to twist (see \cite[p. 400]{Gee94} again) this cannot happen. We see finally that the derived group of $H^\circ$ must have rank 2, and therefore that $H$ is equal to $\GL_3$.)

There exist a weight cocharacter $w$, a Hodge--Tate cocharacter $\mu$, and a complex conjugation $c$ of the form
\[ w(t) = \diag(t^2, t^2, t^2), \]
\[ \mu(t) = \diag(t^2, t, 1), \]
\[ \rho(c) = \left( \begin{array}{ccc} 0 & 0 & 1 \\ 0 & -1 & 0 \\ 1 & 0 & 0 \end{array}\right). \]
Note in particular that $\Ad \rho(c) \circ \mu = w \mu^{-1}$. No twist of $\rho$ is odd, because the odd involutions of $\GL_3$ are outer. The representation $\rho \otimes \epsilon$ is pure of weight 0 and has trivial determinant.

Let $H = \GL_3$, and let $G$ denote the special orthogonal group defined by the matrix
\[ J = \left( \begin{array}{ccc} & & I_3 \\ & 1 & \\ I_3 & & \end{array}\right). \]
We write $R : H \to G$ for the embedding given by $g \mapsto \diag(g, 1, {}^t g^{-1})$. A calculation shows that if $\chi$ is an odd character, then $R \circ (\rho \otimes \chi)$ is odd, in the sense of Definition \ref{def_odd_galois_representation}. Thus if $\chi$ is a geometric odd character, then $R \circ (\rho \otimes \chi)$ is geometric and odd. We note that arguing as in \cite{Har16} or \cite{Sch15}, we should be able to exhibit (the pseudocharacter of) any twist $R \circ (\rho \otimes \chi)$  by an odd geometric character of sufficiently large Hodge--Tate weight as a $p$-adic limit of (pseudocharacters of) $G$-valued representations of Zariski dense image attached to cusp forms on $\Sp_{6}$ with square-integrable archimedean component. Since passing to a $p$-adic limit preserves the conjugacy class of complex conjugation, the oddness of $R \circ (\rho \otimes \chi)$  is a necessary condition for this to be possible. (We note that the oddness of the Galois representations attached to regular $L$-algebraic cusp forms on $\Sp_6$, which is a consequence of the conjectures formulated in \cite{Buz14}, follows from the results of Ta\"ibi \cite{Tai16}.)

However, we cannot conclude that $\rho \otimes \chi$ is odd using Principle \ref{principle_oddness}, because any such twist of $\rho$ will fail one of the conditions there. If $\chi$ is not pure of weight $-2$, then $R \circ (\rho \otimes \chi)$ will not be pure. If $\chi$ is pure of weight $-2$ (for example, if $\chi = \epsilon$), then $R \circ (\rho \otimes \chi)$ will be pure of weight 0, but the Hodge--Tate cocharacter of $R \circ (\rho \otimes \chi)$ will not be regular.

\section{Negative results}\label{sec_positive_results}

Let us fix a prime $p$ and an isomorphism $\iota : \overline{\bbQ}_p \to \bbC$. In this section, we prove the following result (Theorem \ref{thm_intro_automorphy} of the introduction):
\begin{theorem}\label{thm_automorphy}
There exist infinitely pairs $(L, \Pi)$ satisfying the following conditions:
\begin{enumerate}
\item $L \subset \bbC$ is a CM number field and $\Pi$ is a regular $L$-algebraic cuspidal automorphic representation of $\GL_n(\bbA_L)$ such that $\Pi^c \not\cong \Pi^\vee \otimes \chi$ for any character $\chi : L^\times \backslash \bbA_L^\times \to \bbC^\times$.
\item There exists a Shimura datum $(G, X)$ of reflex field $L$ such that the associated Shimura varieties $\Sh_K(G, X)$ are proper and $r_{p, \iota}(\Pi)$ appears as a subquotient of $H^\ast_{\text{\'et}}(\Sh_K(G, X)_{\overline{\bbQ}}, \cF_{\tau,p})$ for some algebraic local system $\cF_\tau$ and for some neat open compact subgroup $K \subset G(\bbA_\bbQ^\infty)$.
\end{enumerate}
\end{theorem}
Let $q$ be an odd prime, and let $K$ be a CM number field containing an imaginary quadratic field.  Fix a CM type $\Phi_K$ of $K$. If $K' / K$ is any CM extension, then we write $\Phi_{K'}$ for the induced CM type. Let $E / K$ be a cyclic CM extension of degree $q^2$, and let $E_0$ denote the unique intermediate subfield of $E / K$.
\begin{lemma}\label{lem_existence_of_CSD_character}
Fix integers $(n_\tau)_{\tau \in \Phi_{E}}$. Then we can find a character $\psi : E^\times \backslash \bbA_E^\times \to \bbC^\times$ and a finite place $v$ of $K$ split over $K^+$ and inert in $E$, all satisfying the following conditions:
\begin{enumerate}\item $\psi \circ \mathbf{N}_{E / E^+} = \| \cdot \|^{1-q^2}$ and $\psi|_{(E \otimes_{E,\tau} \bbC)^\times}(z) = z^{n_\tau} \overline{z}^{1 - q^2 - n_\tau}$ for all $\tau \in \Phi_{E}$.
\item Let $w$ denote the unique place of $E$ lying above $v$. Then for each $g \in \Gal(E / K)$, we have $\psi|_{E_w^\times} \neq \psi^g|_{E_w^\times}$.
\end{enumerate}
\end{lemma}
\begin{proof}
This is a special case of \cite[Lemma A.2.5]{Bar14}.
\end{proof}
We now fix a tuple of integers $(n_\tau)_{\tau \in \Phi_{E}}$ with the following properties:
\begin{itemize}
\item For all $\tau, \tau' \in \Phi_E$ such that $\tau \neq \tau'$, we have $| n_\tau - n_{\tau'} | > 1$.
\item For all $\tau_0, \tau_0' \in \Phi_{E_0}$ such that $\tau_0 \neq \tau_0'$, we have 
\[ \sum_{\substack{\tau \in \Phi_E \\ \tau|_{E_0} = \tau_0}} n_\tau \neq  \sum_{\substack{\tau' \in \Phi_E \\ \tau'|_{E_0} = \tau'_0}} n_{\tau'}. \]
\item There exists $\tau \in \Phi_E$ such that the matrix $(n_{\tau g h})_{g, h \in \Gal(E/K)}$ has non-zero determinant (note that this is a circulant matrix).
\end{itemize}
Fix $\psi$ as in Lemma \ref{lem_existence_of_CSD_character}. Let $\psi_p = r_{p, \iota}(\psi) : \Gamma_E \to \overline{\bbQ}_p^\times$.
\begin{lemma}
\begin{enumerate} \item The representation $\rho_p = \Ind_{\Gamma_E}^{\Gamma_K} \psi_p$ is absolutely irreducible, and there is a regular $L$-algebraic cuspidal automorphic representation $\sigma$ of $\GL_{q^2}(\bbA_K)$ such that $r_{p, \iota}(\sigma) \cong \rho_p$.
\item The representation $\rho_p|_{\Gamma_{K_v}}$ is absolutely irreducible, and $\sigma_v$ is a supercuspidal representation of $\GL_{q^2}(K_v)$.
\end{enumerate} 
\end{lemma}
\begin{proof}
The irreducibility of $\rho_p$ is equivalent to the following statement: for all $g, h \in \Gal(E/K)$ such that $g \neq h$, $\psi_p^g \neq \psi_p^h$; or for all $g, h \in \Gal(E/K)$ such that $g \neq h$, $\psi^g \neq \psi^h$. This statement is true because it is true after restricting $\psi$ to $(E \otimes_\bbQ  \bbR)^\times \subset \bbA_E^\times$. The existence of $\sigma$ follows from the results of \cite[Theorem 4.2]{Art89a}, and $\sigma$ is cuspidal for the same reason that $\rho_p$ is irreducible: see \cite[Corollary 6.5]{Art89a}. The second part is similar.
\end{proof}
Let $H$ denote the Zariski closure of $\rho_p(\Gamma_K)$ in $\GL_{q^2}(\overline{\bbQ}_p)$, and let $\rho_H : \Gamma_K \to H(\overline{\bbQ}_p)$ denote the tautological representation. The group $H$ sits in a short exact sequence
\[ \xymatrix@1{ 1 \ar[r] & \bbG_m^{\Gal(E/K)} \ar[r] & H \ar[r] & \Gal(E/K) \ar[r] & 1. } \]
(To ensure that the image of $\rho_p$ is Zariski dense in this group, we are using the condition imposed above that the matrix $(n_{\tau g h})_{g, h \in \Gal(E/K)}$ has non-zero determinant.) Recall that $E_0$ is the unique intermediate subfield of $E / K$, and consider the Hecke character $\chi = \psi|_{\bbA_{E_0}^\times}$. Let $\chi_p = r_{p, \iota}(\chi) : \Gamma_{E_0} \to \overline{\bbQ}_p^\times$. Let $H_0$ denote the pre-image in $H$ of $\Gal(E / E_0)$. We can find a character $x : H_0 \to \bbG_m$ such that $x \circ \rho_H|_{\Gamma_{E_0}} = \chi_p$. We can find another character $y : H_0 \to \bbG_m$ such that $y \circ \rho_H|_{\Gamma_{E_0}} = \varphi_p$ is a non-trivial character $\Gal(E / E_0) \to \overline{\bbQ}_p^\times$. Let $R = \Ind_{H_0}^H (x \otimes y)$. Then $R$ is a $q$-dimensional representation of the group $H$ and we have $R \circ \rho_H \cong \Ind_{\Gamma_{E_0}}^{\Gamma_K} (\chi_p \otimes \varphi_p)$. 
\begin{proposition}
With notation as above, the representation $r_p = R \circ \rho_H$ has the following properties:
\begin{enumerate}
\item It is absolutely irreducible and Hodge--Tate regular.
\item There exists a cuspidal, regular $L$-algebraic automorphic representation $\pi$ of $\GL_q(\bbA_K)$ such that $r_{p, \iota}(\pi) \cong r_p$.
\item There does not exist a character $\lambda : \Gamma_K \to \overline{\bbQ}_p^\times$ such that $r_p^c \cong r_p^\vee \otimes \lambda$.
\end{enumerate}
\end{proposition}
\begin{proof}
If $\tau_0 : E_0 \hookrightarrow \overline{\bbQ}_p$ is an embedding, let $m_{\tau_0} = \mathrm{HT}_{\tau_0}(\chi)$. Then we have
\[ m_{\tau_0} = \sum_{\substack{\tau : E \hookrightarrow \overline{\bbQ}_p \\ \tau|_{E_0} = \tau_0}} n_\tau. \]
In particular, we see that the $m_{\tau_0}$, $\tau_0 \in \Phi_{E_0}$, are pairwise distinct, and that the representation $r_p$ is Hodge--Tate regular. This representation is irreducible because the conjugates $(\chi_p \otimes \varphi_p)^g$ are pairwise distinct as $g \in \Gal(E_0 / K)$ varies: in fact, these characters already have distinct Hodge--Tate weights. The existence of $\pi$ is again a consequence of \cite[Theorem 4.2]{Art89a}.

It remains to show that $r_p$ is not conjugate self-dual up to twist. Let $\lambda : \Gamma_K \to \overline{\bbQ}_p^\times$ be a character, and suppose that $r_p^c \cong r_p^\vee \otimes \lambda$. Looking at determinants, we see that for each embedding $\tau : K \hookrightarrow \overline{\bbQ}_p$, we have $\mathrm{HT}_\tau(\lambda) = q(1-q)^2$. Restricting to $\Gamma_{E_0}$, we see that there exists $g \in \Gal(E_0 / K)$ such that $\chi_p^c \otimes \varphi_p^c = (\chi_p^\vee \otimes \varphi_p^\vee)^g \otimes \lambda|_{\Gamma_{E_0}}$. Passing to Hodge--Tate weights, this gives for any $\tau_0 : E_0 \hookrightarrow \overline{\bbQ}_p$:
\[ m_{\tau_0 c } + m_{\tau_0 g} = q(1-q^2). \]
Since we also have $m_{\tau_0} + m_{\tau_0 c} = q (1-q^2)$, we find $m_{\tau_0} = m_{\tau_0 g}$, hence $g = 1$ (using again Hodge--Tate regularity of $r_p$). This forces 
\[ \lambda|_{\Gamma_{E_0}} = \chi_p \chi_p^c \varphi_p \varphi_p^c = \epsilon^{q(q^2-1)} \varphi_p \varphi_p^c = \epsilon^{q(q^2-1)} \varphi_p^2. \]
(Note that $\varphi = \varphi^c$ because $\varphi$ factors through the Galois group of a CM extension of the CM field $K$.) However, the character $\varphi_p^2$ does not extend to $\Gamma_K$ (otherwise $E / K$ would have Galois group $(\bbZ / q \bbZ)^2$). This contradiction completes the proof.
\end{proof}
We now apply the following general result.
\begin{proposition}\label{prop_zariski_closure}
Let $\Gamma$ be a profinite group, and let $\rho : \Gamma \to \GL_n(\overline{\bbQ}_p)$ be a continuous semisimple representation. Let $H$ denote the Zariski closure of $\rho(\Gamma)$, and let $R : H \to \GL(V)$ be a finite-dimensional irreducible representation. Then there exist integers $a, b \geq 0$ such that $R \circ \rho$ occurs as a subquotient of $\rho^{\otimes a} \otimes (\rho^\vee)^{\otimes b}$.
\end{proposition}
\begin{proof}
The group $H$ is reductive as it has a faithful semisimple representation (because $\rho$ is semisimple). Let $r$ denote the tautological faithful representation of $H$ on $V$. It then suffices to find integers $a,b \geq 0$ such that $R$ occurs in $r^{\otimes a} \otimes (r^{\vee})^{\otimes b}$. This is presumably standard, see e.g. \cite[Proposition 3.1(a)]{Del82}.
\end{proof}
 By Proposition \ref{prop_zariski_closure}, we can find integers $a, b \geq 0$ such that $r_p$ appears as a subquotient of $\rho_p^{\otimes a} \otimes (\rho_p^\vee)^{\otimes b}$. Let $\ell$ be the residue characteristic of the place $v$. We now fix a cyclic totally real extension $L_0 / \bbQ$ of prime degree $d > a + b$ and in which $\ell$ splits, and set $L = K \cdot L_0$. We observe that this implies the following:
\begin{enumerate}
\item The base change $\pi_L$ is cuspidal, and $r_p|_{\Gamma_L}$ is irreducible and still not conjugate self-dual up to twist. Indeed, $\Gal(E/K)$ is linearly disjoint from $\Gal(L / K)$, so we can just run the above arguments again with $\psi_p|_{\Gamma_L}$ instead of $\psi_p$.
\item The place $v$ splits in $L$, so that if $w$ is a place of $L$ dividing $v$, then $\sigma_{L, w}$ is supercuspidal. 
\end{enumerate}
Let $\Sigma = \sigma_L$ and $\Pi = \pi_L$. We have now almost completed the proof of Theorem \ref{thm_automorphy}: we have constructed, from the data of the extension $E / K$ and the character $\psi$, an automorphic representation $\Pi$ which is regular $L$-algebraic and cuspidal but \emph{not} conjugate self-dual up to twist. In order to complete the proof, we must show that there exists $\tau \in \Phi_L$ and a Shimura datum $(G, X)$ of reflex field $\tau(L)$ such that a twist of $r_{p, \iota}(\Pi)$ by a geometric character appears as a subquotient of $H^\ast_{\text{\'et}}(\Sh_K(G, X)_{\overline{\bbQ}}, \cF_p)$ for some choice of algebraic local system $\cF$.

Fix an embedding $\tau_0 \in \Phi_K$ and disjoint subsets $\Sigma_0, \Sigma_1 \subset \Phi_L$ of embeddings extending $\tau_0$, such that $ | \Sigma_0 | = a$ and $| \Sigma_1 | = b$. Suppose given the following data:
\begin{enumerate}
\item A division algebra $D$ over $L$ of rank $n = q^2$ and centre $L$, together with an involution $\ast : D \to D$ such that $\ast|_L = c$.
\item A homomorphism $h_0 : \bbC \to D \otimes_\bbQ \bbR$ of $\bbR$-algebras such that $h_0(z)^\ast = h_0(\overline{z})$ for all $z \in \bbC$.
\end{enumerate}
Then we can associate to $(D, \ast, h_0)$ a unitary group $G_{00}$ over $L^+$, its restriction of scalars $G_0 = \Res^{L^+}_\bbQ G_{00}$, a unitary similitude group $G$ over $\bbQ$ containing $G_0$, and a Shimura datum $(G, X)$ (see \cite[\S 1]{Kot92} for details). We can choose this data so that the following conditions are satisfied (cf. \cite[\S 2]{Clo91}):
\begin{enumerate}
\item At each place $w | v$ of $L$, $D_w$ is a division algebra of invariant $1 / n$. At each place $w \nmid v v^c$ of $L$, $D$ is split and the group $G_{0, w|_{L^+}}$ is quasi-split.
\item For each $\tau \in \Sigma_0$, we have $n(\tau) = n-1$. For each $\tau \in \Sigma_1$, we have $n(\tau) = 1$. For every other $\tau \in \Phi_L$, we have $n(\tau) = 0$.
\end{enumerate}
The integers $n(\tau)$ here are as on \cite[p. 655]{Kot92}; the second condition here means that we have an isomorphism
\[ G_{0, \bbR} \cong U(n-1, 1)^a \times U(1, n-1)^b \times U(0, n)^{[L^+ : \bbQ] - a - b}. \]
We note that the reflex field of $(G, X)$ is equal to $\tau(L)$, for any $\tau \in \Sigma_0$.

The automorphic representation $\Sigma \otimes \| \det \|^{\frac{1-q^2}{2}}$ is RACSDC\footnote{Regular algebraic, conjugate self-dual, cuspidal, cf. \cite{Clo08}} and descends to an automorphic representation $\Sigma_{G_0}$ of $G_0(\bbA)$ with $\Sigma_{G_0, \infty}$ essentially square integrable and of strictly regular infinitesimal character. (This follows from e.g. the main theorems of \cite{Kal17}. Since we are dealing here with a `simple' Shimura variety, it is possible to prove the existence of this descent much more easily, along the same lines as in the proof of \cite[Proposition 2.3]{Clo93}, making appropriate changes to deal with the presence of more than one non-compact factor at infinity.) Arguing as in the proofs of \cite[Theorem VI.2.9]{Har01} and \cite[Lemma VI.2.10]{Har01}, we can extend $\Sigma_{G_0}$ to a representation $\Sigma_G$ of $G(\bbA_\bbQ)$ such that the integer $a(\Sigma_{G}^\infty)$ of \cite{Kot92} is non-zero. We can therefore apply \cite[Theorem 1]{Kot92} to conclude that there is an algebraic local system $\cF$ such that the $\Sigma_G^\infty$-part of $H^\ast_{\text{\'et}}(\Sh_K(G, X)_{\overline{\bbQ}}, \cF_{p})$ is isomorphic to a character twist of the representation
\[ (\rho_p^{\otimes a} \otimes (\rho_p^\vee)^{\otimes b})^{|a(\Sigma_{G}^\infty)|}|_{\Gamma_L}. \]
 In particular, it admits a twist of the representation $r_p|_{\Gamma_L} = r_{p, \iota}(\Pi)$ by a geometric character as a subquotient. This completes the proof of Theorem \ref{thm_automorphy}. (It is clear that we can generate infinitely many pairs $(L, \Pi)$ just by varying our initial choices.)
\section{Conjectures on Galois representations}\label{sec_conjectures_on_galois_representations}

Let $G$ be a reductive group over $\bbQ$, and let ${}^L G$ be its $L$-group. In order to avoid a proliferation of subscripts, we will in this section fix a prime $p$ and an isomorphism $\iota : \overline{\bbQ}_p \to \bbC$, and write ${}^L G$ also for ${}^L G(\bbC)$, ${}^L G(\overline{\bbQ}_p)$, ${}^L G_\bbC$ and ${}^L G_{\overline{\bbQ}_p}$. We hope that in each case it will be clear from the context exactly which of these groups is intended. In order to analyse the cohomology of Shimura varieties in the next section, we introduce the formalism of the Langlands group, local and global $L$-parameters, and finally local and global $A$-parameters. We will make predictions using these ideas, and state precise conjectures which are independent of the existence of the Langlands group.

Following \cite{Art02}, the global Langlands group should be a locally compact topological group which is an extension
\[ \xymatrix@1{ 1 \ar[r] & K_\bbQ \ar[r] & L_\bbQ \ar[r] & W_\bbQ \ar[r] & 1, } \]
where $W_\bbQ$ is the Weil group of $\bbQ$. For each place $v$ of $\bbQ$, there should be a continuous embedding $L_{\bbQ_v} \to L_\bbQ$, defined up to conjugacy, where $L_{\bbQ_v}$ is the local Langlands group:
\[ L_{\bbQ_v} = \left\{ \begin{array}{cc} W_{\bbQ_v} \times \mathrm{SU}_2(\bbR) & v \text{ non-archimedean;} \\ W_{\bbQ_v} & v \text{ archimedean.} \end{array}\right. \]
The irreducible $n$-dimensional continuous complex representations of the group $L_\bbQ$ should be in bijection with the cuspidal automorphic representations of $\GL_n(\bbA_\bbQ)$. More generally, if $\pi$ is an essentially tempered automorphic representation of $G(\bbA_\bbQ)$, then one expects that there should be a corresponding continuous homomorphism $\phi : L_\bbQ \to {}^L G$ with the property that for each place $v$ of $\bbQ$, $\pi_v$ is in the $L$-packet corresponding to $\phi|_{L_{\bbQ_v}}$. (To formulate this statement supposes that the local Langlands correspondence for $G(\bbQ_v)$ is known. It thus has an unconditional sense at least if either $v$ is archimedean, or $v$ is non-archimedean and $\pi_v$ is unramified.) The homomorphism $\phi$ should be an $L$-parameter, i.e.\ it should be semisimple, and the projection $L_\bbQ \to \pi_0({}^L G)$ should factor through the canonical surjection $L_\bbQ \to \Gamma_\bbQ \to \pi_0({}^L G)$. The condition that $\pi$ is essentially tempered should imply that the image of $\phi$ is essentially bounded, i.e.\ bounded modulo the centre of ${}^L G$.

Let $\underline{W} = (\bbG_m \times \bbG_m) \rtimes \{ 1, c \}$, where $c$ acts by swapping factors. Then there is an embedding $L_\bbR \to \underline{W}(\bbC)$, which sends $z$ to $(z, \overline{z}) \rtimes 1$ and $j$ to $(-i, -i) \rtimes c$. We say that a homomorphism $\phi_\infty : L_\bbR \to {}^L G$ is $L$-algebraic if it is the restriction to $L_\bbR$ of a map $\underline{W}(\bbC) \to {}^L G(\bbC)$ which comes from a morphism $\underline{W}_\bbC \to {}^L G_\bbC$ of algebraic groups. In this case we write $a_{\phi_\infty} : \underline{W}_\bbC \to {}^L G_\bbC$ for the corresponding morphism of algebraic groups, and call it the algebraic $L$-parameter corresponding to $\phi_\infty$. We say that an irreducible admissible representation of $G(\bbR)$ is $L$-algebraic if its Langlands parameter is $L$-algebraic, and that an automorphic representation $\pi$ of $G(\bbA_\bbQ)$ is $L$-algebraic if $\pi_\infty$ is. 

We say that an $L$-parameter $\phi : L_\bbQ \to {}^L G$ is $L$-algebraic if $\phi|_{L_{\bbR}}$ is $L$-algebraic. Langlands has suggested that there should be a morphism from $L_\bbQ$ to the motivic Galois group of $\bbQ$ (with $\bbC$-coefficients). Based on the conjectures in \cite{Buz14} (see also \cite[\S 6]{Art02}), one can guess that a morphism $\phi : L_\bbQ \to {}^L G$ factors through the motivic Galois group if and only if $\phi$ is $L$-algebraic. Passing to $p$-adic realizations, and bearing in mind the discussion at the end of \S \ref{sec_galois_representations}, this leads us to predict that for any $L$-algebraic morphism $\phi : L_\bbQ \to {}^L G$, there exists a continuous homomorphism $\rho_\phi : \Gamma_\bbQ \to {}^L G$ satisfying the following conditions:
\begin{enumerate}
\item $\rho_\phi$ is geometric and mixed. (We recall from Definition \ref{def_mixed} that the weight cocharacter $w$ is then defined.)
\item For each prime $l$, $\mathrm{WD}(\rho_\phi|_{\Gamma_{\bbQ_l}})$ is $\widehat{G}$-conjugate to $\phi|_{L_{\bbQ_l}}$.\footnote{Here we write $\mathrm{WD}$ for the Weil--Deligne representation associated to a $p$-adic representation of $\Gamma_{\bbQ_l}$, assumed to be de Rham if $l = p$. See for example \cite{Tat79} in the case $l \neq p$, or \cite[\S 2.2]{Bre02} in the case $l = p$. We will soon restrict to unramified places in order to avoid any unnecessary complications.}
\item There exists a Hodge--Tate cocharacter $\mu$ of $\rho_\phi$ and a complex conjugation $c \in \Gamma_\bbQ$ such that $\Ad(\rho_\phi(c)) \circ \mu = \mu^{-1} w$ and the morphism $a_\rho : \underline{W} \to {}^L G$, $(z_1, z_2) \mapsto z_1^\mu z_2^{w - \mu}$, $c \mapsto \rho_\phi(c)$ is $\widehat{G}$-conjugate to $a_{\phi|_{L_{\bbR}}}$.
\end{enumerate}
(Here, as in \S \ref{sec_galois_representations}, we write $w : \bbG_m \to Z({}^L G)$ for the weight cocharacter of the Galois representation $\rho_\phi$.) This leads us to the following conjecture:
\begin{conjecture}\label{conj_galois_reps_for_tempered_forms}
Let $\pi$ be an essentially tempered automorphic representation of $G(\bbA_\bbQ)$ which is $L$-algebraic. Then there exists a continuous homomorphism $\rho_\pi : \Gamma_\bbQ \to {}^L G$ satisfying the following conditions:
\begin{enumerate}
\item $\rho_\pi$ is geometric and pure. 
\item For each prime $l \neq p$ such that $\pi_l$ is unramified, $\rho_\pi|_{\Gamma_{\bbQ_l}}$ is unramified and $\rho_\pi(\Frob_l)$ is $\widehat{G}$-conjugate to the Satake parameter of $ \pi_l$.
\item There exists a Hodge--Tate cocharacter $\mu$ of $\rho_\pi$ and a complex conjugation $c \in \Gamma_\bbQ$ such that $\Ad \rho_\pi(c) \circ \mu = \mu^{-1} w$ and the morphism $a_\rho : \underline{W} \to {}^L G$, $(z_1, z_2) \mapsto z_1^\mu z_2^{w - \mu}$, $c \mapsto \rho_\pi(c)$ is $\widehat{G}$-conjugate to $a_{\pi_\infty}$, the algebraic $L$-parameter of $\pi_\infty$.
\end{enumerate}
\end{conjecture}
We note that this conjecture makes no reference to the Langlands group. It is worth comparing this conjecture with those made in \cite[\S 3.2]{Buz14}. In \emph{loc. cit.}, the authors do not restrict to essentially tempered automorphic representations, imposing instead only $L$-algebraicity. At infinity, they predict only the $\widehat{G}$-conjugacy class of $\rho_\pi(c)$ in ${}^L G$. By contrast, we are predicting both the conjugacy class of $\rho_\pi(c)$ and the existence of a Hodge--Tate cocharacter that is compatible with $\rho_\pi(c)$, in some sense. This is motivated by the discussion at the end of \S \ref{sec_galois_representations}. We note that this stronger prediction would be false without the restriction that $\pi$ is essentially tempered, as one sees either by considering holomorphic Eisenstein series for $\GL_2$ or holomorphic Saito--Kurakawa lifts on $\mathrm{PSp}_4$ (cf. \cite[\S 3]{Lan79}). 
\begin{proposition}\label{prop_oddness_tempered_case}
Let $\pi$ be an essentially tempered $L$-algebraic automorphic representation of $G(\bbA_\bbQ)$, and suppose that $\pi_\infty$ is essentially square-integrable. Suppose that Conjecture \ref{conj_galois_reps_for_tempered_forms} holds for $\pi$. Then:
\begin{enumerate}
\item $\Ad \rho_\pi(c)$ is an odd involution of $\widehat{G}$, in the sense of Definition \ref{def_odd_involution}.
\item Let $H_\pi$ denote the Zariski closure of the image of $\rho_\pi$. Then $\Ad \rho_\pi(c)$ is an odd involution of $H_\pi^\circ$.
\end{enumerate}
\end{proposition}
\begin{proof}
We note that our definition of Hodge--Tate cocharacter implies that $\mu$ in fact factors through $H_\pi$. Therefore the second part of the proposition will follow from the first part and from Lemma \ref{lem_oddness_in_reductive_subgroups} if we can establish the first part and at the same time show that $\mu$ is a regular cocharacter of $\widehat{G}$.

To prove the whole proposition, it therefore suffices to show that $a_{\pi_\infty}$ has the property that $a_{\pi_\infty}|_{\bbG_m \times 1}$ is a regular cocharacter and $a_{\pi_\infty}(c)$ acts as $-1$ on $\Cent(\widehat{G}, a_{\pi_\infty}(\bbG_m \times 1)) / Z(\widehat{G})$. To see this, we just describe the $L$-parameters of the $L$-algebraic discrete series representations of $G(\bbR)$. Fix a choice of maximal torus $T \subset G_\bbR$ which is compact mod centre. Let $B \subset G_\bbC$ be a Borel subgroup containing $T_\bbC$. Let $\widehat{T} \subset \widehat{B} \subset \widehat{G}$ be the corresponding maximal torus and Borel subgroups of the dual group. Let $\sigma_T$ denote the $L$-action of complex conjugation on $\widehat{T}$ corresponding to the given real structure on $T$. Then the restriction of the $L$-action of $\widehat{G}$ to $\widehat{T}$ is given by $\Ad n_G \circ \sigma_T$, where $n_G \in N(\widehat{G}, \widehat{T})$ represents the longest element of $W(\widehat{G}, \widehat{T})$ with respect to the system of positive roots given by $\widehat{B}$.

If $\pi_\infty$ is an $L$-algebraic essentially discrete series representation of $G(\bbR)$, then, after possibly replacing $a_{\pi_\infty}$ by a $\widehat{G}$-conjugate, $a_{\pi_\infty}$ is given by the formula
\[ (z_1, z_2) \mapsto z_1^\mu z_2^{w - \mu}, \]
\[ c \mapsto n_G z \rtimes c \]
for some element $z$ of the centre $Z(\widehat{G})$ of $\widehat{G}$, and some cocharacters $\mu \in X_\ast(\widehat{T})$, $w \in X_\ast(Z(\widehat{G}))$, such that $\mu$ is regular and dominant with respect to $\widehat{B}$. In particular, $\Ad a_{\pi_\infty}(c)$ acts as $\sigma_T$ on $\widehat{T}$. If $\widehat{T}^\text{ad}$ denotes the image of $\widehat{T}$ in $\widehat{G}^\text{ad}$, then $\sigma_T$ acts as $-1$ on $\widehat{T}^\text{ad}$. It follows that $\Ad a_{\pi_\infty}(c)$ is an odd involution of $\widehat{G}$, as desired.
\end{proof}
We note that a calculation of the type appearing in Proposition \ref{prop_oddness_tempered_case} has appeared already in the note of Gross \cite{Gro}. We now turn to the question of generalizing this proposition to non-tempered automorphic representations. We will do this just for discrete automorphic representations, using Arthur's formalism of $A$-parameters. By definition, a global $A$-parameter is a continuous semisimple homomorphism
\[ \psi : L_\bbQ \times \SL_2 \to {}^L G \]
such that the induced map $L_\bbQ \to \pi_0({}^L G)$ factors through the canonical one $\Gamma_\bbQ \to \pi_0({}^L G)$, and
with the property that $\psi|_{L_{\bbQ}}$ is essentially bounded. To any $A$-parameter $\psi$ we can associate an $L$-parameter $\phi_\psi$, given by the formula
\[ \phi_\psi(w) = \psi( w, \diag( \| w \|^{1/2}, \| w \|^{-1/2} ) ), \]
where $\| \cdot \| : W_\bbQ \to \bbR_{>0}$ is the norm pulled back from the idele class group. We define a local $A$-parameter similarly to be a continuous semisimple homomorphism
\[ \psi_v : L_{\bbQ_v} \times \SL_2 \to {}^L G \]
such that the induced map $L_{\bbQ_v} \to \pi_0({}^L G)$ factors through the canonical one $\Gamma_{\bbQ_v} \to \pi_0({}^L G)$, and with the property that $\psi_v|_{L_{\bbQ_v}}$ is essentially bounded. Arthur's conjectures predict that for any local $A$-parameter $\psi_v$, one should be able to define a set (called an $A$-packet) $\Pi_{\psi_v}$ of representations of $G(\bbQ_v)$, containing the $L$-packet of $\phi_{\psi_v}$. To any discrete automorphic representation $\pi$ of $G(\bbA_\bbQ)$, one should be able to associate a global $A$-parameter $\psi : L_\bbQ \times \SL_2 \to {}^L G$ with the property that for each place $v$ of $\bbQ$, $\pi_v \in \Pi_{\psi_v}$. 

We now discuss what this has to do with Galois representations. Let $\psi$ be a global $A$-parameter such that $\phi_{\psi}$ is $L$-algebraic. Let $(x, t, y)$ be the $\frs\frl_2$-triple in $\widehat{\frg}$ determined by $\psi|_{\SL_2}$, and let $M_1$ denote the centralizer in $\widehat{G}$ of this $\frs\frl_2$-triple. Let $M_1'$ denote the centralizer in ${}^L G$ of this $\frs\frl_2$-triple.  Let $M = M_1 \cdot \lambda_t(\bbG_m)$, where $\lambda_t : \bbG_m \to \SL_2 \to \widehat{G}$ is the cocharacter with derivative $t$, and let $M' = M_1 \cdot \lambda_t(\bbG_m)$. Then there are exact sequences
\[ \xymatrix@1{ 1 \ar[r] & M_1 \ar[r] & M_1' \ar[r] & \pi_0({}^L G) \ar[r] & 1 } \]
and
\[ \xymatrix@1{ 1 \ar[r] & M \ar[r] & M' \ar[r] & \pi_0({}^L G) \ar[r] & 1, } \]
and $\phi_\psi$ factors through a homomorphism $\phi' : L_\bbQ \to M'$. This leads us to expect the existence of a representation $\rho_\psi : \Gamma_\bbQ \to {}^L G$ and an $\frs\frl_2$-triple $(x, t, y)$ in $\widehat{\frg}$ satisfying the following conditions: 
\begin{enumerate}
\item $\rho_\psi$ is geometric and mixed. Moreover, $dw - t \in \frz(\widehat{\frg})$ and the following formulae hold: for each $\gamma \in \Gamma_\bbQ$,
\[ \Ad \rho_\psi(\gamma) (x) = \epsilon^{-1}(\gamma)x, \Ad \rho_\psi(\gamma)(t) = t, \Ad \rho_\psi(\gamma)(y) = \epsilon(\gamma)y. \]
\item For almost all primes $l$, $\mathrm{WD}(\rho_\psi|_{\Gamma_{\bbQ_l}})$ is $\widehat{G}$-conjugate to $\phi_\psi|_{L_{\bbQ_l}}$.
\item There exists a Hodge--Tate cocharacter $\mu$ of $\rho_\psi$ and a complex conjugation $c \in \Gamma_\bbQ$ with the following properties: we have $\Ad \rho_\psi(c) \circ \mu = \mu^{-1} w$. Define $a_\rho : \underline{W} \to {}^L G$, $(z_1, z_2) \mapsto z_1^\mu z_2^{w - \mu}$, $c \mapsto \rho_\psi(c)$. Then the pair $(a_\rho, (x, t, y))$ is $\widehat{G}$-conjugate to $(a_{\phi_\psi|_{L_\bbR}}, (d\psi(x_0), d\psi(t_0), d\psi(y_0)))$.
\end{enumerate}
In contrast to the case of tempered representations, we do not see a way to phrase this prediction solely in terms of automorphic representations, without making reference to $A$-parameters. However, one can make a conjecture supposing only that the local $A$-packets at infinity have been defined. A definition has been given in \cite{Ada92}. 

In order to avoid unnecessary complications here, we will now assume for the rest of \S \ref{sec_conjectures_on_galois_representations} that $G_\bbR$ contains a maximal torus $T$ which is compact modulo centre. Let $t_0 \in T(\bbR)$ be such that $\Ad(t_0)$ induces a Cartan involution of $G_\bbR^\text{ad}$. Let $K_\infty = G(\bbR)^{t_0}$. We will state a conjecture only for automorphic representations which are (up to twist) cohomological, in the sense that there exists an irreducible algebraic representation $\tau$ of $G_\bbC$ such that $H^\ast(\frg_\bbC, K_\infty; \pi_\infty \otimes \tau) \neq 0$. In this case we can use the $A$-parameters and packets described by Adams--Johnson (see \cite{Ada87} and also \cite[\S 5]{Art89}). The representations in these packets were later shown by Vogan and Zuckerman to be the unitary cohomological representations of $G(\bbR)$ \cite{Vog84, Vog84a}.
\begin{conjecture}\label{conj_galois_reps_for_non_tempered_representations}
Let $\pi$ be a discrete $L$-algebraic automorphic representation of $G(\bbA_\bbQ)$ such that a twist of $\pi_\infty$ is cohomological. Then there exists an $\frs\frl_2$-triple $(x, t, y)$ in $\widehat{\frg}$ and a continuous representation $\rho_\pi : \Gamma_\bbQ \to {}^L G$ satisfying the following conditions:
\begin{enumerate}
\item $\rho_\pi$ is geometric and mixed. Moreover, $dw - t \in \frz(\widehat{\frg})$ and the following formulae hold: for each $\gamma \in \Gamma_\bbQ$,
\[ \Ad \rho_\pi(\gamma) (x) = \epsilon^{-1}(\gamma)x, \Ad \rho_\pi(\gamma)(t) = t, \Ad \rho_\pi(\gamma)(y) = \epsilon(\gamma)y. \]
\item For almost all primes $l \neq p$ such that $\pi_l$ is unramified, $\rho_\pi|_{\Gamma_{\bbQ_l}}$ is unramified and $\rho_\pi(\Frob_l)$ is $\widehat{G}$-conjugate to the Satake parameter of $ \pi_l$.
\item There exists a Hodge--Tate cocharacter $\mu$ of $\rho_\pi$, a complex conjugation $c \in \Gamma_\bbQ$, and an $A$-parameter $\psi : L_\bbR \times \SL_2 \to {}^L G$ which is up to twist of the type described in \cite[\S 5]{Art89}, with the following property: we have $\Ad \rho_\pi(c) \circ \mu = \mu^{-1} w$, and $\phi_\psi$ is $L$-algebraic. Define $a_\rho : \underline{W} \to {}^L G$, $(z_1, z_2) \mapsto z_1^\mu z_2^{w - \mu}$, $c \mapsto \rho_\pi(c)$. Then the pair $(a_\rho, (x, t, y))$ is $\widehat{G}$-conjugate to $(a_{\phi_\psi}, d\psi(x_0), d\psi(t_0), d\psi(y_0))$.
\end{enumerate}
\end{conjecture}
This leads us to the following generalization of Proposition \ref{prop_oddness_tempered_case}:
\begin{proposition}\label{prop_oddness_nontempered_case}
Let $\pi$ be a discrete $L$-algebraic automorphic representation of $G(\bbA_\bbQ)$ such that some twist of $\pi_\infty$ is cohomological. Suppose that $\pi$ satisfies Conjecture \ref{conj_galois_reps_for_non_tempered_representations}. Let $H_\pi \subset {}^L G$ denote the Zariski closure of the image of $\rho_\pi$. Then $\Ad \rho_\pi(c)$ is an odd involution of $H_\pi^\circ$.
\end{proposition}
This generalizes the second part of Proposition \ref{prop_oddness_tempered_case} because if $\sigma$ is an essentially square-integrable representation of $G(\bbR)$, then $G_\bbR$ contains a maximal torus which is compact mod centre and some twist of $\sigma$ is cohomological. One can also generalize the first part of Proposition \ref{prop_oddness_tempered_case}, although we don't prove this here as we don't need it. 
\begin{proof}
By Lemma \ref{lem_oddness_in_reductive_subgroups} and the final requirement of Conjecture \ref{conj_galois_reps_for_non_tempered_representations}, it will again suffice to show that if $\psi : L_\bbR \times \SL_2 \to {}^L G$ is a twist of one of the $A$-parameters considered in \cite[\S 5]{Art89} such that $\phi_\psi$ is $L$-algebraic, then the cocharacter $a_{\phi_\psi}|_{\bbG_m \times 1}$ is regular in the group $ M_1^\circ \cdot \lambda_t(\bbG_m)$, where $M_1 =  \text{Cent}(\widehat{G}, \SL_2)$, and $\Ad a_{\phi_\psi}(c)$ induces an odd involution of $M_1^\circ$. 

Using Arthur's explicit description in \cite[\S 5]{Art89}, we see that the pair $(a_{\phi_\psi}, (d \psi(x_0), d \psi(t_0), d \psi(y_0)))$ has the following form. First let $T \subset G_\bbR$ be a maximal torus which is compact mod centre, and let $B \subset G_\bbC$ be a Borel subgroup containing $T$. Let $\widehat{T} \subset \widehat{B} \subset \widehat{G}$ be the corresponding maximal torus and Borel subgroup of the dual group. Then there is a standard Levi subgroup $\widehat{L}$ such that $(x, t, y)$ is a principal $\frs\frl_2$-triple in $\widehat{\frl}$, and $a_{\phi_\psi}$ is given by the formula
\[ (z_1, z_2) \mapsto z_1^{\mu} z_2^{w - \mu}, \]
\[ c \mapsto w(-i) n_L^{-1} n_G z \rtimes c, \]
for some $z \in Z(\widehat{G})$ and some regular dominant cocharacter $\mu \in X_\ast(\widehat{T})$. (We note that this differs from \cite[p. 30]{Art89}, where $n_L^{-1} n_G$ is replaced by $n_G n_L^{-1}$, but our choice appears to give a correct formula. We recall that $n_G, n_L$ are elements of the derived groups of $\widehat{G}$ and $\widehat{L}$, respectively, which represent the longest elements of the respective Weyl groups with respect to the sets of positive roots determined by $\widehat{B}$.) There is a decomposition $\widehat{\frg} = \widehat{\frl} \oplus \widehat{\frn}^+ \oplus \widehat{\frn}^-$, where $\widehat{\frn}^+$ and $\widehat{\frn}^-$ are the Lie algebras of the unipotent radicals of, respectively, the standard parabolic subgroup containing $\widehat{L}$, and its opposite. This decomposition is $\widehat{\frl}$-invariant, hence $\frs\frl_2$-invarant. It follows that the Lie algebra of $\ffrm_1$ admits a similar decomposition 
\[ \ffrm_1 = \fra \oplus \ffrm_1^+  \oplus \ffrm_1^-, \]
where $\fra$ is the Lie algebra of the connected centre $A$ of $\widehat{L}$. In fact $A$ is a maximal torus of $M_1$: indeed, we have $\widehat{L} = \text{Cent}(\widehat{G}, A)$, so any torus of $M_1$ containing $A$ is necessarily contained in $\text{Cent}(\widehat{L}, \psi(\SL_2))$, hence in $A$. The proof will be complete if we can show that $\Ad a_{\phi_\psi(c)}$ acts as $-1$ on the image of $A$ in the adjoint group of $M_1^\circ$. However, since $\Ad a_{\phi_\psi(c)} = \Ad (w(-i) n_L^{-1} n_G) \circ \sigma_G$, we see that $\Ad a_{\phi_\psi(c)}$ acts on  $\widehat{T}$ as $\Ad(n_L^{-1}) \circ \sigma_T$. The action on $A$ is therefore equal to $\sigma_T|_A$, which has the desired property (recall that $\sigma_T$ acts as $-1$ on $\widehat{T} / Z(\widehat{G})$).
\end{proof}
We end this section by stating a variant of Conjecture \ref{conj_galois_reps_for_non_tempered_representations} for discrete automorphic representations $\pi$ such that $\pi_\infty$ is cohomological (not just up to twist). This variant is based on the observation that $\pi$ is then necessarily $C$-algebraic, in the sense of \cite{Buz14}. 

In \cite[Proposition 4.3.1]{Buz14}, Buzzard and Gee define a canonical extension of $G$ by $\bbG_m$:
\[ \xymatrix@1{ 1 \ar[r] & \bbG_m \ar[r] & \widetilde{G} \ar[r] & G \ar[r]& 1. } \]
They define the $C$-group ${}^C G$ to be the $L$-group ${}^C G = {}^L \widetilde{G}$. The virtue of the $C$-group is that if $\sigma$ is an irreducible admissible representation of $G(\bbR)$ which is cohomological, then its pullback $\sigma'$ to $\widetilde{G}(\bbR)$ has a canonical twist which is $L$-algebraic (see the proof of \cite[Proposition 5.3.6]{Buz14}). Moreover, the group ${}^C G$ admits an explicit description: its dual group can be identified as the quotient
\begin{equation}\label{eqn_dual_group_of_C_group} \widehat{\widetilde{G}} = (\bbG_m \times \widehat{G}) / (-1, \chi(-1)), 
\end{equation}
where $\chi \in X_\ast(\widehat{T})$ is the sum of the positive roots with respect to some choice of Borel containing $\widehat{T}$ (the element $\chi(-1)$ is central and does not depend on the choice of Borel). This leads us to the following conjecture for automorphic forms on $G$, which follows from applying Conjecture \ref{conj_galois_reps_for_non_tempered_representations} to suitable $L$-algebraic discrete automorphic representations of $\widetilde{G}$:
\begin{conjecture}\label{conj_representations_in_the_C_group}
Let $\pi$ be a discrete automorphic representation of $G(\bbA_\bbQ)$ such that $\pi_\infty$ is cohomological. Then there exists an $\frs\frl_2$-triple $(x, t, y)$ in $\widehat{\widetilde{\frg}}$ and a continuous representation $\rho'_\pi : \Gamma_\bbQ \to {}^C G$ satisfying the following conditions:
\begin{enumerate}
\item $\rho'_\pi$ is geometric and mixed. Moreover, $dw - t \in \frz(\widehat{\widetilde{\frg}})$ and the following formulae hold: for each $\gamma \in \Gamma_\bbQ$,
\[ \Ad \rho'_\pi(\gamma) (x) = \epsilon^{-1}(\gamma)x, \Ad \rho'_\pi(\gamma)(t) = t, \Ad \rho'_\pi(\gamma)(y) = \epsilon(\gamma)y. \]
\item For almost all primes $l \neq p$ such that $\pi_l$ is unramified, $\rho'_\pi|_{\Gamma_{\bbQ_l}}$ is unramified and $\rho'_\pi(\Frob_l)$ is $\widehat{\widetilde{G}}$-conjugate to the image of $(l^{-1/2}, t(\pi_l))$, where $t(\pi_l)$ is the Satake parameter of $\pi_l$.
\item There exists a Hodge--Tate cocharacter $\mu$ of $\rho'_\pi$, a complex conjugation $c \in \Gamma_\bbQ$, and an $A$-parameter $\psi : L_\bbR \times \SL_2 \to {}^C G$ which is up to twist of the type described in \cite[\S 5]{Art89}, with the following property: we have $\Ad \rho'_\psi(c) \circ \mu = \mu^{-1} w$. Define $a_\rho : \underline{W} \to {}^C G$, $(z_1, z_2) \mapsto z_1^\mu z_2^{w - \mu}$, $c \mapsto \rho'_\pi(c)$. Then the pair $(a_\rho, (x, t, y))$ is $\widehat{\widetilde{G}}$-conjugate to $(a_{\phi_\psi},( d\psi(x_0), d\psi(t_0), d\psi(y_0)))$.
\end{enumerate}
\end{conjecture}

\section{Intersection cohomology of Shimura varieties}\label{sec_negative_results}

In this section we will use conjectures of Arthur and Kottwitz to give a conjectural description of the Galois representations appearing in the intersection cohomology of Shimura varieties. Let $(G, X)$ be a Shimura datum. We write $E$ for the reflex field and $\Sh_K(G, X)$ for the associated Shimura variety over $E$ with respect to some neat open compact subgroup $K \subset G(\bbA_\bbQ^\infty)$. Let $\cF_\tau$ be an algebraic local system. We fix a prime $p$ and an isomorphism $\iota : \overline{\bbQ}_p \to \bbC$.

 Let $j :\Sh_K(G, X) \to \Sh^\text{min}_K(G, X)$ denote the open embedding of $\Sh_K(G, X)$ into its  minimal compactification. The intersection cohomology groups 
\[ IH_{\tau, K} = H^\ast_{\text{\'et}}(\Sh^\text{min}_K(G, X)_{\overline{\bbQ}}, j_{! \ast}\cF_{\tau, p}) \]
are $\cH(G(\bbA_\bbQ^\infty),K) \otimes \overline{\bbQ}_p [\Gamma_E]$-modules, where $\cH$ denotes the usual convolution Hecke algebra, and are finite-dimensional as $\overline{\bbQ}_p$-vector spaces. Our aim is to understand the irreducible subquotient $\overline{\bbQ}_p[\Gamma_E]$-modules of $H^\ast_{\text{\'et}}(\Sh^\text{min}_K(G, X)_{\overline{\bbQ}}, j_{! \ast}\cF_{\tau, p})$. If $\pi^\infty$ is an irreducible admissible $\bbC[G(\bbA^\infty)]$-module, then we define
\[ W_p(\pi^\infty) = \Hom_{G(\bbA_\bbQ^\infty)}(\iota^{-1} \pi_\infty, \lim_{K} IH_{\tau, K}). \]
Then $W_p(\pi^\infty)$ is a finite-dimensional $\overline{\bbQ}_p[\Gamma_E]$-module and each irreducible subquotient of $H^\ast_{\text{\'et}}(\Sh^\text{min}_K(G, X)_{\overline{\bbQ}}, j_{! \ast} \cF_{\tau, p})$ is isomorphic to an irreducible subquotient of some  $W_p(\pi^\infty)$.

The action of $G(\bbA_\bbQ^\infty)$ on intersection cohomology can be understood in terms of discrete automorphic representations. Indeed, the Zucker conjecture (proved independently by Looijenga and Saper--Stern) states that $IH_{\tau, K}$ can be identified with the $L^2$-cohomology $H^\ast_{(2)}(\Sh_K(G, X)(\bbC), \cF_{\tau} \otimes_{\overline{\bbQ}} \bbC)$, which admits a decomposition
\[ H^\ast_{(2)}(\Sh_K(G, X)(\bbC), \cF_{\tau} \otimes_{\overline{\bbQ}} \bbC) = \bigoplus_\pi m(\pi) (\pi^\infty)^K \otimes H^\ast(\frg_\bbC, K_1; \pi_\infty \otimes \tau), \]
where $K_1 = \Cent(G(\bbR),h)$ for some choice of $h \in X$. This led Kottwitz to give a conjectural description of the space $W_p(\pi^\infty)$ in terms of the $A$-parameters giving rise to $\pi^\infty$ \cite{Kot90}. The paper \cite{Joh13} gave a reformulation of Kottwitz's conjecture in terms of the $C$-group. We now discuss some consequences of the reformulated conjectures in \cite{Joh13} for the spaces $W_p(\pi^\infty)$. 

In order to describe the contribution of $\pi$ to cohomology, we recall (\cite[\S 1]{Joh13}) that the Shimura datum $(G, X)$ determines an irreducible representation $r_C : {}^C (G_E) \to \GL(V)$, where $V$ is a finite-dimensional complex vector space and $G_E$ denotes base extension to the reflex field $E$. Then (see \cite[Conjecture 8]{Joh13}) Kottwitz\footnote{Kottwitz makes this prediction under additional assumptions on $G$, namely that the derived group $G^\text{der}$ is simply connected and the maximal $\bbR$-split subtorus of $Z(G)$ is $\bbQ$-split. However, it seems natural to predict this in general, and since our discussion is conjectural we will do this in order not to impose further conditions on $G$.}  predicts the existence of an isomorphism
\[ W_p(\pi^\infty) = \bigoplus_{{\psi :  \pi^\infty \in \Pi_{\psi^\infty}}} (U_\pi \otimes (r_C \circ \rho'_\pi|_{\Gamma_E}))_{\epsilon_\psi}, \]
where the direct sum runs over the set of $A$-parameters giving rise to $\pi^\infty$. The precise meaning of the symbols $U_\pi$ and $\epsilon_\psi$ is not important for us here; rather, we need only the following weaker consequence of this prediction, which we can again phrase independently of the existence of the Langlands group. (We note that, according to \cite[p. 59]{Art89}, if $W_p(\pi^\infty) \neq 0$ then there should exist $\pi_\infty$ such that $\pi = \pi^\infty \otimes \pi_\infty$ is a discrete automorphic representation of $G(\bbA_\bbQ)$ and a twist of $\pi_\infty$ is cohomological, in the sense of \S \ref{sec_conjectures_on_galois_representations}.)
\begin{conjecture}\label{conj_consequence_of_Arthur_Kottwitz}
Let $\pi$ be a discrete automorphic representation of $G(\bbA_\bbQ)$ such that $W_p(\pi^\infty) \neq 0$. Then there exists an $\frs\frl_2$-triple $(x, t, y)$ in $\widehat{\widetilde{\frg}}$ and a continuous representation $\rho'_\pi : \Gamma_\bbQ \to {}^C G(\overline{\bbQ}_p)$ as in Conjecture \ref{conj_representations_in_the_C_group}, and each irreducible subquotient $\overline{\bbQ}_p[\Gamma_E]$-module of $W_p(\pi^\infty)$ is isomorphic to an irreducible subquotient of the representation $r_C \circ \rho'_\pi|_{\Gamma_E}$.
\end{conjecture}
\begin{proposition}\label{prop_self_duality}
Let $\pi$ be a discrete automorphic representation of $G(\bbA_\bbQ)$ such that $W_p(\pi^\infty) \neq 0$. Suppose that Conjecture \ref{conj_consequence_of_Arthur_Kottwitz} holds for $\pi$, and let $V$ be an irreducible subquotient $\overline{\bbQ}_p[\Gamma_E]$-module of $W_p(\pi^\infty)$. Let $H_V \subset \GL(V)$ denote the Zariski closure of the image of $\Gamma_E$ and $\rho_V : \Gamma_E \to H_V(\overline{\bbQ}_p)$ the tautological representation. Let $\theta$ be an odd involution of $H_V^\circ$. Let $c \in \Gamma_{\bbQ}$ be a choice of complex conjugation. Then there exists a finite Galois extension $M / E$ such that $\rho_V(\Gamma_M) \subset H_V^\circ(\overline{\bbQ}_p)$ and the representations $\rho_V^c|_{\Gamma_M}$, $\theta \circ \rho_V|_{\Gamma_M}$ are $H_V^\circ(\overline{\bbQ}_p)$-conjugate modulo $Z_{H_V^\circ}$.
\end{proposition}
\begin{proof}
Let $f : H_V^\circ \to H_V^{\circ ,\text{ad}}$ denote the projection to the adjoint group. We must show that there exists a finite Galois extension $M / E$ such that $\rho_V(\Gamma_M) \subset H_V^\circ(\overline{\bbQ}_p)$ and the representations $f \circ \rho_V^c|_{\Gamma_M}$, $f \circ \theta \circ \rho_V|_{\Gamma_M}$ are $H_V^{\circ, \text{ad}}$-conjugate.  

Let $H_\bbQ$ denote the Zariski closure of $\rho'_\pi(\Gamma_\bbQ)$ in ${}^C G$, and let $H_E$ denote the Zariski closure of $\rho'_\pi(\Gamma_E)$. Then $H_E \subset H_\bbQ$ and $H_E^\circ = H_\bbQ^\circ$. There is a map $H_E \to \GL(V)$, and $H_V$ is equal to the image of this map. Let $r_E : \Gamma_E \to H_E(\overline{\bbQ}_p)$ be the tautological representation. According to Proposition \ref{prop_oddness_nontempered_case}, $\Ad(\rho'_\pi(c))$ induces an odd involution of $H_E^\circ$, which we denote by $\theta'$. Since $\theta'$ is odd, it must leave invariant the simple factors of $H_E^{\circ, \text{ad}}$, and therefore the kernel of the map $H_E^{\circ,\text{ad}} \to H_V^{\circ,\text{ad}}$. Replacing $\theta$ with a conjugate, we can therefore assume without loss generality that the map $H_E^{\circ,\text{ad}} \to H_V^{\circ,\text{ad}}$ intertwines $(\theta')^\text{ad}$ and $\theta^\text{ad}$.

We have $r_E^c = \Ad(\rho'_\pi(c)) \circ r_E$. Let $M / E$ be any Galois extension such that $r_E(\Gamma_M) \subset H_E^\circ$. Then $r_E^c|_{\Gamma_M} = \theta' \circ r_E|_{\Gamma_M}$. Pushing this identity down to $H_V^{\circ, \text{ad}}$ gives the identity $f \circ \rho_V^c|_{\Gamma_M} = \theta^\text{ad} \circ f \circ \rho_V|_{\Gamma_M} = f \circ \theta \circ \rho_V|_{\Gamma_M}$, which is what we needed to prove.
\end{proof}
This leads to the following result.
\begin{theorem}
Let $\rho : \Gamma_E \to \GL_n(\overline{\bbQ}_p)$ be a continuous representation which is strongly irreducible, in the sense that for any finite extension $M / E$, $\rho|_{\Gamma_M}$ is irreducible. Let $\pi$ be a discrete automorphic representation of $G(\bbA_\bbQ)$, and suppose that Conjecture \ref{conj_consequence_of_Arthur_Kottwitz} holds for $\pi$. Then if $\rho$ is isomorphic to a subquotient of $W_p(\pi^\infty)$, then $\rho$ is conjugate self-dual up to twist.
\end{theorem}
\begin{proof}
Let $H_V$ denote the Zariski closure of the image of $\rho$ in $\GL_n = \GL(V)$, and let $\theta$ be a Chevalley involution of $H_V^\circ$. Let $\rho_V : \Gamma_E \to H_V(\overline{\bbQ}_p)$ denote the tautological representation. Our hypotheses imply that $V$ is an irreducible representation of $H_V^\circ$, and $V \circ \theta  \cong V^\vee$. Proposition \ref{prop_self_duality} implies that there is a finite Galois extension $M / E$ such that $\rho_V(\Gamma_M) \subset H_V^\circ(\overline{\bbQ}_p)$ and $\rho_V^c |_{\Gamma_M}$ and $\theta \circ \rho_V|_{\Gamma_M}$ are $H_V^\circ(\overline{\bbQ}_p)$-conjugate modulo the centre of $H_V^\circ$. 

It follows that there is a continuous character $\chi_M : \Gamma_M \to \overline{\bbQ}_p^\times$ such that $\rho^c|_{\Gamma_M}$ and $\rho^\vee|_{\Gamma_M} \otimes\chi_M$ are $\GL_n(\overline{\bbQ}_p)$-conjugate. Lemma \ref{lem_twists} implies that $\chi_M$ extends to a character $\chi : \Gamma_E \to \overline{\bbQ}_p^\times$ such that $\rho^c \cong \rho^\vee \otimes \chi$. This completes the proof.
\end{proof}
\begin{lemma}\label{lem_twists}
Let $\Gamma$ be a group, let $\Delta \subset \Gamma$ be a finite index normal subgroup, and let $r_1, r_2 : \Gamma \to \GL_n(\overline{\bbQ}_p)$ be representations such that $r_1|_\Delta = r_2|_\Delta \otimes \chi$ for some character $\chi : \Delta \to \overline{\bbQ}_p^\times$. Suppose that for any finite index subgroup $\Delta' \subset \Gamma$, $r_1|_{\Delta'}$ is irreducible. Then $\chi$ extends to a character $\chi' : \Gamma \to \overline{\bbQ}_p$ such that $r_1 = r_2 \otimes \chi'$.
\end{lemma}
\begin{proof}
If $\sigma \in \Gamma$, $\tau \in \Delta$, we have
\[ r_1(\sigma \tau \sigma^{-1}) = r_1(\sigma) r_1(\tau) r_1(\sigma^{-1}) = r_1(\sigma) r_2(\tau) \chi(\tau) r_1(\sigma^{-1}) \]
and
\[  r_1(\sigma \tau \sigma^{-1}) = r_2(\sigma) r_2(\tau) r_2(\sigma^{-1}) \chi(\sigma \tau \sigma^{-1}). \]
It follows that if $g_\sigma = r_2(\sigma)^{-1}r_1(\sigma)$, then for all $\tau \in \Delta$ we have
\[ g_\sigma r_2(\tau) g_\sigma^{-1} = r_2(\tau) \chi^\sigma(\tau) \chi(\tau)^{-1}. \]
There are two cases. If $\chi^\sigma = \chi$ for all $\sigma \in \Gamma$, then the irreducibility of $r_2|_{\Delta}$ implies that $g_\sigma = \lambda_\sigma \in \overline{\bbQ}_p^\times$, and we are done on defining $\chi'(\sigma) = \lambda_\sigma$.

Otherwise, there exists $\sigma \in \Gamma$ such that $\chi^\sigma \neq \chi$. Then $r_2 \cong r_2 \otimes \chi^\sigma \otimes \chi^{-1}$. The character $\chi^\sigma / \chi$ has order dividing $n$, and if $\Delta' = \ker \chi^\sigma / \chi$ then $r_2|_{\Delta'}$ must be reducible, contradicting our assumptions. This completes the proof.
\end{proof}

\section{Cohomology of open Shimura varieties}\label{open}

The previous section deals with the question of Galois representations appearing as subquotients in the intersection cohomology of minimal compactifications of Shimura varieties. It is of course natural to ask what, if anything, changes if one considers other types of cohomology groups. The most natural choices here are cohomology and compactly supported cohomology (with values in an automorphic $\ol{\bbQ}_{p}$-local system) of the open Shimura varieties. These two are related by Poincar\'e duality, so it suffices to consider ordinary cohomology. The expectation is then that the Galois representations occurring as subquotients of the cohomology of open Shimura varieties are exactly the same as those occurring in the intersection cohomology of minimal compactifications. 

At least under some assumptions on the Shimura datum, this is known and due to Morel. Since the precise statement in this form is not (as far as we know) in the literature, we sketch one way of deducing it from work of Morel \cite{Mor08, Mor10, Mor}, which we learnt from a talk of Morel at the Institute for Advanced Study \cite{Mor11}. Any mistakes below are entirely due to the authors. See also work of Nair for another approach \cite{Nai}.

The proof relies of Morel's theory of weight truncations and Pink's formula. In \cite{Hub97}, Huber constructs a triangulated category, which we will denote by $D_{m}(Y,\ol{\bbQ}_{p})$, of so-called horizontal mixed complexes of constructible $\ol{\bbQ}_{p}$-sheaves on any separated scheme $Y$ of finite type over any number field $F$. Informally speaking, `horizontal' means that the complexes extend to complexes of constructible $\ol{\bbQ}_{p}$-sheaves on some flat model of $Y$ over some open subset $U$ of $\Spec \cO_{F}$. A horizontal complex is pure if its specializations to all but finitely many closed points of $U$ are pure (with constant weight), and mixed complexes are those arising as extensions of pure complexes. $D_{m}(Y,\ol{\bbQ}_{p})$ admits a perverse t-structure, whose core $Perv_{m}(Y)$ is the so-called category of horizontal mixed perverse sheaves on $Y$. We refer to \cite[\S 1-3]{Hub97} for precise definitions. In \cite{Mor}, Morel considers the full subcategory $\ms{M}(Y)$ of $Perv_{m}(Y)$ consisting of objects admitting a weight filtration; see \cite[Definition 3.7]{Hub97} for the definition of a weight filtration on objects of $Perv_{m}(Y)$. A weight filtration is unique if it exists and morphisms between two objects admitting weight filtrations are strict \cite[Lemma 3.8]{Hub97}. Morel proves that $\ms{M}(Y)$ and its bounded derived category $D^{b}(\ms{M}(Y))$ satisfies a long list of properties (cf. \cite{Mor} Th\'eor\`eme 3.2 and Proposition 6.1), including stability under six operations and Tate twists and that it contains $\ol{\bbQ}_{p}[d]$ if $Y$ is smooth and pure of dimension $d$. Moreover, all objects in $D^{b}(\ms{M}(Y))$ admit weight filtrations and one may define weight truncations $w_{\leq a}$, $w_{> a}$ ($a\in \bbZ \cup \{ \pm \infty \}$) as in \cite{Mor08}; see \cite[\S 8]{Mor}. 

We recall some notation for Shimura varieties and their minimal compactifications used in the previous section, and then set up some more notation. We let $G$ be a connected reductive group over $\bbQ$ and $(G,X)$ a Shimura datum where we allow $G^{ad}$ to have simple factors of compact type over $\bbQ$ (to allow zero-dimensional Shimura varieties), and we write $E$ for the reflex field of $(G,X)$. If $K\sub G(\bbA_{f})$ is a neat compact open subgroup, we let $\Sh_{K}=\Sh_{K}(G,X)$ denote the canonical model (over $E$) of the complex Shimura variety of level $K$. We let $\Sh_{K}^{\min}=\Sh_{K}^{\min}(G,X)$ denote its minimal compactification and we write $j : \Sh_{K} \to \Sh_{K}^{\min}$ for the open embedding. Recall that a parabolic subgroup $P \sub G$ defined over $\bbQ$ is called admissible if its image in every simple fact $G^{\prime}$ of $G^{ad}$ is either $G^{\prime}$ or a maximal proper parabolic of $G^{\prime}$. We write $N_{P}$ for the unipotent radical of $P$, $U_{P}$ for the center of $N_{P}$, and $M_{P}=P/N_{P}$ for the Levi quotient. We let $X_{P}$ be the boundary component corresponding to $P$. Following \cite[p. 2--3]{Mor10}, we make the following additional assumptions on $(G,X)$: First, assume that $G^{ad}$ is simple. Second, for every admissible parabolic $P$ of $G$, there exist connected reductive subgroups $L_{P}$ and $G_{P}$ of $M_{P}$ such that

\begin{itemize}
\item $M_{P}$ is the direct product of $L_{P}$ and $G_{P}$;

\item $G_{P}$ contains a certain normal subgroup $G_{1}$ of $M_{P}$ defined by Pink \cite[(3.6)]{Pin92}, and the quotient $G_{P}/Z(G_{P})G_{1}$ is $\bbR$-anisotropic;

\item $L_{P}\sub Cent_{M_{P}}(U_{P}) \sub Z(M_{P})L_{P}$;

\item $G_{P}(\bbR)$ acts transitively on $X_{P}$, and $L_{P}(\bbR)$ acts trivially on $X_{P}$;

\item for every neat compact open subgroup $K_{M}\sub M_{P}(\bbA_{f})$, $K_{M}\cap L_{P}(\bbQ)=K_{M} \cap Cent_{M_{P}(\bbQ)}(X_{P})$.
\end{itemize}

If $G$ satisfies these assumptions, then $G_{P}$ satisfies these assumptions for any admissible parabolic $P$ of $G$ \cite[Remark 1.1.1]{Mor10}. These assumptions are satisfied if $G$ is an inner form of a unitary similitude or symplectic simlitude group \cite[Example 1.1.2]{Mor10}. Let us briefly describe the stratification on $\Sh_{K}^{\min}$ (under these assumptions). Let $P$ be an admissible parabolic of $G$ and let $Q_{P}\sub P$ be the preimage of $G_{P}$. There is a Shimura datum $(G_{P}, X_{P})$ with reflex field $E$. Let $g \in G(\bbA_{f})$. Set $H_{P}=gKg^{-1}\cap P(\bbQ)Q_{P}(\bbA_{f})$, $H_{L}=gKg^{-1}\cap L_{P}(\bbQ)N_{P}(\bbA_{f})$, $K_{Q}=gKg^{-1}\cap Q_{P}(\bbA_{f})$, and $K_{N}=gKg^{-1}\cap N_{P}(\bbA_{f})$. Then there is a morphism
$$ \Sh_{K_{Q}/K_{N}}(G_{P},X_{P}) \to \Sh_{K}^{\min} $$
which is finite over its image. The group $H_{P}$ acts on $\Sh_{K_{Q}/K_{N}}(G_{P},X_{P})$ and the action is trivial on the normal subgroup $H_{L}K_{Q}$; moreover, $H_{P}/H_{L}K_{Q}$ is finite. Quotienting out by the action of $H_{P}$ gives a locally closed immersion
$$ i_{P,g} : \Sh_{K_{Q}/K_{N}}(G_{P},X_{P})/H_{P} \to \Sh_{K}^{\min}, $$
which extends to a finite morphism
$$ \ol{i}_{P,g} : \Sh_{K_{Q}/K_{N}}^{\min}(G_{P},X_{P})/H_{P} \to \Sh_{K}^{\min}. $$
The boundary of $\Sh_{K}^{\min}$ is the union of the images of the $i_{P,g}$ for proper admissible parabolics $P$ and elements $g\in G(\bbA_{f})$. If $P,P^{\prime}$ are admissible parabolics and $g,g^{\prime}\in G(\bbA)_{f}$, then the images of $i_{P,g}$ and $i_{P^{\prime},g^{\prime}}$ are either equal or disjoint, and they are equal if and only if there exists a $\gamma \in G(\bbQ)$ such that $P^{\prime}=\gamma P \gamma^{-1}$ and $P(\bbQ)Q_{P}(\bbA_{f})gK=P(\bbQ)Q_{P}(\bbA_{f})\gamma^{-1}g^{\prime}K$.

In view of this, we fix a minimal parabolic subgroup $P_{0}$ of $G$ and let $P_{1},\dots,P_{n}$ be the admissible parabolics of $G$ containing $P_{0}$, where the order is defined by $r\leq s$ if and only if $U_{P_{r}}\sub U_{P_{s}}$. We simplify the notation and write $N_{r}=N_{P_{r}}$, $i_{r,g}=i_{P_{r},g}$ etc. Then, the boundary of $\Sh_{K}^{\min}$ is the union of the images of the $i_{r,g}$ for $r=1,\dots,n$ and $g\in G(\bbA_{f})$, and $i_{r,g}$ and $i_{s,h}$ have the same image if and only if $r=s$ and there is a $\gamma\in G(\bbQ)$ such that $P_{r}(\bbQ)Q_{r}(\bbA_{f})gK=P_{r}(\bbQ)Q_{r}(\bbA_{f})\gamma^{-1}hK$. For a fixed $r$, put 
$$ \Sh_{K,r}=\bigcup_{g\in G(\bbA_{f})}{\rm Im}(i_{r,g}).$$
The $\Sh_{K,r}$ are locally closed subvarieties of $\Sh_{K}^{\min}$ and the collection $\Sh_{K,1},\dots,\Sh_{K,n}$ defines a stratification of $\Sh^{\min}_{K}$ in the sense of \cite[D\'efinition 3.3.1]{Mor08}. We let $i_{r}: \Sh_{K,r} \to \Sh_{K}^{\min}$ denote the inclusion. By \cite[Corollaire 8.1.4]{Mor}, all results of \cite[\S 3]{Mor08} go through in the setting of \cite[\S 7]{Mor}. In particular, following \cite[Proposition 3.3.4]{Mor08}, we may define ${}^{w}D^{\leq \ul{a}}= {}^{w}D^{\leq \ul{a}}(\ms{M}(\Sh_{K}^{\min}))$ (resp. ${}^{w}D^{> \ul{a}}={}^{w}D^{> \ul{a}}(\ms{M}(\Sh_{K}^{\min}))$) for any $\ul{a}=(a_{1},\dots,a_{n}) \in (\bbZ \cup \{\pm \infty\})^{n}$ to be the full subcategory of $D^{b}(\ms{M}(\Sh_{K}^{\min}))$ of objects $C$ such that $i_{r}^{\ast}C\in D^{b}(\ms{M}(\Sh_{K,r}))$ (resp. $i_{r}^{!}C\in D^{b}(\ms{M}(\Sh_{K,r}))$) has weights $\leq a_{r}$ (resp. $>a_{r}$) for all $r\in \{1,\dots,n\}$. Then $({}^{w}D^{\leq \ul{a}}, {}^{w}D^{> \ul{a}})$ defines a t-structure on $D^{b}(\ms{M}(\Sh_{K}^{\min}))$ and we get weight truncation functors $w_{\leq \ul{a}}$ and $w_{> \ul{a}}$ for all $\ul{a} \in (\bbZ \cup \{\pm \infty\})^{n}$. In particular, for any $r\in \{1,\dots,n\}$ and $a\in \bbZ \cup \{\pm \infty\}$, we get weight truncation functors 
$$w_{\leq a}^{r} := w_{\leq (+\infty,\dots, +\infty,a,+\infty, \dots, +\infty)};$$ 
$$w_{> a}^{r} := w_{>(-\infty,\dots, -\infty,a,-\infty, \dots, -\infty)},$$
where $a$ is in the $r$-th place. 

The final piece of notation and terminology we need concerns automorphic lisse $\ol{\bbQ}_{p}$-sheaves on $\Sh_{K}$. Let $Rep_{G}$ be the (semisimple) abelian category of algebraic representations of $G$ over $\ol{\bbQ}_{p}$ and let $D^{b}(Rep_{G})$ be its bounded derived category. There is an additive triangulated functor
$$ \cF^{K} : D^{b}(Rep_{G}) \to D_{c}^{b}(\Sh_{K},\ol{\bbQ}_{p}) $$
to the derived category of constructible $\ol{\bbQ}_{p}$-sheaves on $\Sh_{K}$ \cite[(1.10)]{Pin92}, \cite[2.1.4]{Mor06}. If $V\in Ob(D^{b}(Rep_{G}))$, then all cohomology sheaves of $\cF^{K}V$ are lisse and in particular perverse up to shift.

Assume now that $(G,X)$ is of abelian type. We may then find a finite set of primes $\Sigma$ of $E$, containing all primes above $p$, such that all objects above extend to $\Spec \cO_{E} \setminus \Sigma$. More precisely, by \cite[Proposition 1.3.4]{Mor10} we may choose $\Sigma$ such that conditions (1)-(7) on \cite[p. 8]{Mor10} are satisfied. In particular, conditions (5) and (7) imply that the functor $\cF^{K}$ may be naturally viewed as a functor
$$ \cF^{K} : D^{b}(Rep_{G}) \to D^{b}(\ms{M}(\Sh_{K})); $$ 
similar remarks apply for all strata of the minimal compactifications. We remark that $Rep_{G}$ has a notion of weight coming from the Shimura datum (see \cite[p.7]{Mor10}) and condition (7) says, in particular, that if $V$ is pure then $\cF^{K}V$ is pure. Note also that condition (6) says that Pink's formula holds for our integral model with the extended (complexes of) sheaves. Let us now state and prove the main theorem of this section.

\begin{theorem}\label{morel} Consider the collection of all Shimura data $(G,X)$ of abelian type which satisfies the list of conditions in the bullet points above, with $G^{ad}$ simple. Then the intersection cohomology groups $H^{\ast}_{\text{\'et}}(\Sh_{K}^{\min}(G,X)_{\overline{\bbQ}},j_{!\ast}\cF^{K}V)$, for all $(G,X)$ as above and all levels $K$ and $V\in D^{b}(Rep_{G})$, contain the same irreducible Galois representation as subquotients as the ordinary cohomology groups $H^{\ast}_{\text{\'et}}(\Sh_{K}(G, X)_{\overline{\bbQ}},\cF^{K}V)$. 
\end{theorem} 

\begin{proof}
We will prove the following claim by induction on $d$: 

\emph{Claim:} For any $d\in \bbZ_{\geq 0}$, the intersection cohomology groups $H^{\ast}_{\text{\'et}}(\Sh_{K}^{\min}(G,X)_{\overline{\bbQ}},j_{!\ast}\cF^{K}V)$, for all $(G,X)$ as in the theorem with $\dim \Sh_{K}(G,X) \leq d$ (and all levels $K$ and $V\in D^{b}(Rep_{G})$), contain the same irreducible Galois representation as subquotients as the ordinary cohomology groups $H^{\ast}_{\text{\'et}}(\Sh_{K}(G, X)_{\overline{\bbQ}},\cF^{K}V)$ for all $(G,X)$ as in theorem with $\dim \Sh_{K}(G,X) \leq d$ (and all levels $K$ and $V\in D^{b}(Rep_{G})$). 

This would clearly give us the theorem. If $d=0$, then $H^{\ast}(\Sh_{K}^{\min},j_{!\ast}\cF^{K}V)=H^{\ast}(\Sh_{K},\cF^{K}V)$ and the assertion is clear. For the induction step, assume $d=\dim \Sh_{K} \geq 1$. Without loss of generality assume that $V$ is concentrated in a single degree and that it is pure. Let $a$ be the weight of $\cV := (\cF^{K}V)[d]$; this is a pure perverse sheaf in $\ms{M}(\Sh_{K})$. By \cite[Proposition 3.3.4]{Mor08} (which holds in our situation by \cite[Corollaire 8.1.4]{Mor}) we have a distinguished triangle
$$ w_{\leq a}^{2}Rj_{\ast}\cV \to Rj_{\ast}\cV \to Ri_{2\ast}w_{>a}i_{2}^{\ast}Rj_{\ast}\cV \to $$
in $D^{b}(\ms{M}(\Sh_{K}^{\min}))$. Applying \cite[Proposition 3.3.4]{Mor08} again to this triangle, we get a square
\[
\xymatrix{ w_{\leq a}^{3}w^{2}_{\leq a}Rj_{\ast}\cV \ar[r]\ar[d] & w_{\leq a}^{3}Rj_{\ast}\cV \ar[r]\ar[d] & w_{\leq a}^{3}Ri_{2\ast}w_{>a}i_{2}^{\ast}Rj_{\ast}\cV \ar[r]\ar[d] & \\
w_{\leq a}^{2}Rj_{\ast}\cV \ar[r]\ar[d] & Rj_{\ast}\cV \ar[r]\ar[d] & Ri_{2\ast}w_{>a}i_{2}^{\ast}Rj_{\ast}\cV \ar[r]\ar[d] & \\
Ri_{3\ast}w_{>a}i_{3}^{\ast}w_{\leq a}^{2}Rj_{\ast}\cV \ar[r]\ar[d] & Ri_{3\ast}w_{>a}i_{3}^{\ast}Rj_{\ast}\cV \ar[r]\ar[d] & Ri_{3\ast}w_{>a}i_{3}^{\ast}Ri_{2\ast}w_{>a}i_{2}^{\ast}Rj_{\ast}\cV \ar[r]\ar[d] & \\
  &  &  &
}
\]
of triangles. Continuing in this way, we get an $(n-1)$-dimensional hypercube of triangles, with `top left corner' $w^{n}_{\leq a}\dots w_{\leq a}^{2}Rj_{\ast}\cV = w_{\leq a}Rj_{\ast}\cV$.\footnote{See \cite[Th\'eor\`eme 3.3.5]{Mor08} for the equality in the Grothendieck group that results from this hypercube.} Since $\cV$ is a pure perverse sheaf of weight $a$, $w_{\leq a}Rj_{\ast}\cV = j_{!\ast}\cV$ by \cite[Th\'eor\`eme 3.1.4]{Mor08}. By \cite[Proposition 1.4.5]{Mor10} (whose proof relies of Pink's formula and the theory of weight truncations, so is valid in our setting), the cohomology of the complexes $Ri_{n_{r}\ast}w_{>a}i_{n_{r}}^{\ast}...Ri_{n_{1}\ast}w_{>a}i_{n_{1}}^{\ast}Rj_{\ast}\cV$
for $n_{1}<...<n_{r}$, for any sequence in $\{2,\dots,n\}$, are subquotients of direct sums of ordinary cohomology groups of automorphic complexes for boundary strata of $\Sh_{K}^{\min}$ (subquotients as the boundary strata are finite Galois quotients of Shimura varieties). As the complexes $Ri_{n_{r}\ast}w_{>a}i_{n_{r}}^{\ast}...Ri_{n_{1}\ast}w_{>a}i_{n_{1}}^{\ast}Rj_{\ast}\cV$ for $n_{1}<...<n_{r}$ together with $Rj_{\ast}\cV$ make up the `lower right' $2 \times \dots \times 2$ hypercube and $j_{!\ast}\cV$ sits in the top left corner, the induction step follows by taking long exact sequences of the triangles in the hypercube. This finishes the proof of the claim.
\end{proof}

To conclude, let us note that we would optimistically expect the theorem to hold with no assumptions on the $(G,X)$, and that the proof would proceed along the same lines in an ideal world; we note that the Hodge-theoretic analogue holds by work of Nair \cite{Nai}. At present, we do not know how one could try to remove the assumption of abelian type (unless one replaces $D^{b}(Rep(G))$ with a smaller subcategory, cf. \cite[Proposition 1.3.4]{Mor10}), since the link to geometry seems necessary to prove that the sheaves $\cF^{K}V$ are pure of the expected weight when $V$ is a pure representation. We would naively suspect that it might be possible to remove the other assumptions, but we have not looked into this.

\bibliographystyle{alpha}
\bibliography{Galois}

\begin{thebibliography}{BLGGT14}

\bibitem[ABV92]{Ada92}
Jeffrey Adams, Dan Barbasch, and David~A. Vogan, Jr.
\newblock {\em The {L}anglands classification and irreducible characters for
  real reductive groups}, volume 104 of {\em Progress in Mathematics}.
\newblock Birkh\"auser Boston, Inc., Boston, MA, 1992.

\bibitem[AC89]{Art89a}
James Arthur and Laurent Clozel.
\newblock {\em Simple algebras, base change, and the advanced theory of the
  trace formula}, volume 120 of {\em Annals of Mathematics Studies}.
\newblock Princeton University Press, Princeton, NJ, 1989.

\bibitem[AJ87]{Ada87}
Jeffrey Adams and Joseph~F. Johnson.
\newblock Endoscopic groups and packets of nontempered representations.
\newblock {\em Compositio Math.}, 64(3):271--309, 1987.

\bibitem[Art89]{Art89}
James Arthur.
\newblock Unipotent automorphic representations: conjectures.
\newblock {\em Ast\'erisque}, (171-172):13--71, 1989.
\newblock Orbites unipotentes et repr\'esentations, II.

\bibitem[Art02]{Art02}
James Arthur.
\newblock A note on the automorphic {L}anglands group.
\newblock {\em Canad. Math. Bull.}, 45(4):466--482, 2002.
\newblock Dedicated to Robert V. Moody.

\bibitem[AV16]{Ada16}
Jeffrey Adams and David~A. Vogan, Jr.
\newblock Contragredient representations and characterizing the local
  {L}anglands correspondence.
\newblock {\em Amer. J. Math.}, 138(3):657--682, 2016.

\bibitem[BG14]{Buz14}
Kevin Buzzard and Toby Gee.
\newblock The conjectural connections between automorphic representations and
  {G}alois representations.
\newblock In {\em Automorphic forms and {G}alois representations. {V}ol. 1},
  volume 414 of {\em London Math. Soc. Lecture Note Ser.}, pages 135--187.
  Cambridge Univ. Press, Cambridge, 2014.

\bibitem[BLGGT14]{Bar14}
Thomas Barnet-Lamb, Toby Gee, David Geraghty, and Richard Taylor.
\newblock Potential automorphy and change of weight.
\newblock {\em Ann. of Math. (2)}, 179(2):501--609, 2014.

\bibitem[BM02]{Bre02}
Christophe Breuil and Ariane M\'ezard.
\newblock Multiplicit\'es modulaires et repr\'esentations de {${\rm GL}_2({\bf
  Z}_p)$} et de {${\rm Gal}(\overline{\bf Q}_p/{\bf Q}_p)$} en {$l=p$}.
\newblock {\em Duke Math. J.}, 115(2):205--310, 2002.
\newblock With an appendix by Guy Henniart.

\bibitem[Cal]{Cal}
Frank Calegari.
\newblock Even {G}alois representations.
\newblock Unpublished lecture notes, available at
  \texttt{http://www.math.uchicago.edu/{$\sim$}fcale/papers/FontaineTalk-Adjusted.pdf}.

\bibitem[Car86]{Car86}
Henri Carayol.
\newblock Sur les repr\'esentations {$l$}-adiques associ\'ees aux formes
  modulaires de {H}ilbert.
\newblock {\em Ann. Sci. \'Ecole Norm. Sup. (4)}, 19(3):409--468, 1986.

\bibitem[CHT08]{Clo08}
Laurent Clozel, Michael Harris, and Richard Taylor.
\newblock Automorphy for some {$l$}-adic lifts of automorphic mod {$l$}
  {G}alois representations.
\newblock {\em Publ. Math. Inst. Hautes \'Etudes Sci.}, (108):1--181, 2008.
\newblock With Appendix A, summarizing unpublished work of Russ Mann, and
  Appendix B by Marie-France Vign\'eras.

\bibitem[Clo90]{Clo90}
Laurent Clozel.
\newblock Motifs et formes automorphes: applications du principe de
  fonctorialit\'e.
\newblock In {\em Automorphic forms, {S}himura varieties, and {$L$}-functions,
  {V}ol.\ {I} ({A}nn {A}rbor, {MI}, 1988)}, volume~10 of {\em Perspect. Math.},
  pages 77--159. Academic Press, Boston, MA, 1990.

\bibitem[Clo91]{Clo91}
Laurent Clozel.
\newblock Repr\'esentations galoisiennes associ\'ees aux repr\'esentations
  automorphes autoduales de {${\rm GL}(n)$}.
\newblock {\em Inst. Hautes \'Etudes Sci. Publ. Math.}, (73):97--145, 1991.

\bibitem[Clo93]{Clo93}
Laurent Clozel.
\newblock On the cohomology of {K}ottwitz's arithmetic varieties.
\newblock {\em Duke Math. J.}, 72(3):757--795, 1993.

\bibitem[Del71]{Del71}
Pierre Deligne.
\newblock Formes modulaires et repr\'esentations {$l$}-adiques.
\newblock In {\em S\'eminaire {B}ourbaki. {V}ol. 1968/69: {E}xpos\'es
  347--363}, volume 175 of {\em Lecture Notes in Math.}, pages Exp.\ No.\ 355,
  139--172. Springer, Berlin, 1971.

\bibitem[Del79a]{Del79a}
P.~Deligne.
\newblock Valeurs de fonctions {$L$}\ et p\'eriodes d'int\'egrales.
\newblock In {\em Automorphic forms, representations and {$L$}-functions
  ({P}roc. {S}ympos. {P}ure {M}ath., {O}regon {S}tate {U}niv., {C}orvallis,
  {O}re., 1977), {P}art 2}, Proc. Sympos. Pure Math., XXXIII, pages 313--346.
  Amer. Math. Soc., Providence, R.I., 1979.
\newblock With an appendix by N. Koblitz and A. Ogus.

\bibitem[Del79b]{Del79}
Pierre Deligne.
\newblock Vari\'et\'es de {S}himura: interpr\'etation modulaire, et techniques
  de construction de mod\`eles canoniques.
\newblock In {\em Automorphic forms, representations and {$L$}-functions
  ({P}roc. {S}ympos. {P}ure {M}ath., {O}regon {S}tate {U}niv., {C}orvallis,
  {O}re., 1977), {P}art 2}, Proc. Sympos. Pure Math., XXXIII, pages 247--289.
  Amer. Math. Soc., Providence, R.I., 1979.

\bibitem[Del82]{Del82}
Pierre Deligne.
\newblock Hodge cycles on abelian varieties.
\newblock In {\em Hodge cycles, motives and {S}himura varieties}, volume 900 of
  {\em Lecture Notes in Math.}, pages 9--100. Springer-Verlag, Berlin, 1982.

\bibitem[Fal88]{Fal88}
Gerd Faltings.
\newblock {$p$}-adic {H}odge theory.
\newblock {\em J. Amer. Math. Soc.}, 1(1):255--299, 1988.

\bibitem[FM95]{Fon95}
Jean-Marc Fontaine and Barry Mazur.
\newblock Geometric {G}alois representations.
\newblock In {\em Elliptic curves, modular forms, \& {F}ermat's last theorem
  ({H}ong {K}ong, 1993)}, Ser. Number Theory, I, pages 41--78. Int. Press,
  Cambridge, MA, 1995.

\bibitem[Gro]{Gro}
Benedict Gross.
\newblock Odd galois representations.
\newblock Unpublished note, available online at
  \texttt{http://www.math.harvard.edu/{$\sim$}gross/preprints/{G}alois\_{R}ep.pdf}.

\bibitem[HLTT16]{Har16}
Michael Harris, Kai-Wen Lan, Richard Taylor, and Jack Thorne.
\newblock On the rigid cohomology of certain {S}himura varieties.
\newblock {\em Res. Math. Sci.}, 3:Paper No. 37, 308, 2016.

\bibitem[HT01]{Har01}
Michael Harris and Richard Taylor.
\newblock {\em The geometry and cohomology of some simple {S}himura varieties},
  volume 151 of {\em Annals of Mathematics Studies}.
\newblock Princeton University Press, Princeton, NJ, 2001.
\newblock With an appendix by Vladimir G. Berkovich.

\bibitem[Hub97]{Hub97}
Annette Huber.
\newblock Mixed pervese sheaves for schemes over number fields.
\newblock {\em Compositio Math.}, 108(1):107--121, 1997.

\bibitem[Joh13]{Joh13}
Christian Johansson.
\newblock A remark on a conjecture of {B}uzzard-{G}ee and the cohomology of
  {S}himura varieties.
\newblock {\em Math. Res. Lett.}, 20(2):279--288, 2013.

\bibitem[KMSW]{Kal17}
Tasho Kaletha, Alberto Minguez, Sug-Woo Shin, and Paul-James White.
\newblock Endoscopic classification of representations: Inner forms of unitary
  groups.
\newblock Preprint, available online at
  \texttt{https://arxiv.org/abs/1409.3731}.

\bibitem[Kot90]{Kot90}
Robert~E. Kottwitz.
\newblock Shimura varieties and {$\lambda$}-adic representations.
\newblock In {\em Automorphic forms, {S}himura varieties, and {$L$}-functions,
  {V}ol.\ {I} ({A}nn {A}rbor, {MI}, 1988)}, volume~10 of {\em Perspect. Math.},
  pages 161--209. Academic Press, Boston, MA, 1990.

\bibitem[Kot92]{Kot92}
Robert~E. Kottwitz.
\newblock On the {$\lambda$}-adic representations associated to some simple
  {S}himura varieties.
\newblock {\em Invent. Math.}, 108(3):653--665, 1992.

\bibitem[Lan79]{Lan79}
R.~P. Langlands.
\newblock Automorphic representations, {S}himura varieties, and motives. {E}in
  {M}\"archen.
\newblock In {\em Automorphic forms, representations and {$L$}-functions
  ({P}roc. {S}ympos. {P}ure {M}ath., {O}regon {S}tate {U}niv., {C}orvallis,
  {O}re., 1977), {P}art 2}, Proc. Sympos. Pure Math., XXXIII, pages 205--246.
  Amer. Math. Soc., Providence, R.I., 1979.

\bibitem[Mil83]{Mil83}
J.~S. Milne.
\newblock The action of an automorphism of {${\bf C}$}\ on a {S}himura variety
  and its special points.
\newblock In {\em Arithmetic and geometry, {V}ol. {I}}, volume~35 of {\em
  Progr. Math.}, pages 239--265. Birkh\"auser Boston, Boston, MA, 1983.

\bibitem[Mor]{Mor}
Sophie Morel.
\newblock Complexes mixtes sure un schema de type fini sur {Q}.
\newblock {\em Available at https://web.math.princeton.edu/~smorel/}.

\bibitem[Mor06]{Mor06}
Sophie Morel.
\newblock {\em Complexes d'intersection des compactifications de
  {B}aily-{B}orel. {L}e cas des groupes unitaires sur {$\mathbb{Q}$}}.
\newblock PhD thesis, Universit\'e Paris-Sud XI - Orsay, 2006.

\bibitem[Mor08]{Mor08}
Sophie Morel.
\newblock Complexes pond\'er\'es sur les compactifications de {B}aily--{B}orel:
  le cas de vari\'eti\'es de {S}iegel.
\newblock {\em J. Amer. Math. Soc.}, 21(1):23--61, 2008.

\bibitem[Mor10]{Mor10}
Sophie Morel.
\newblock {\em On the cohomology of certain noncompact {S}himura varieties.
  With an appendix by {R}obert {K}ottwitz.}, volume 173 of {\em Annals of
  Mathematics Studies}.
\newblock Princeton University Press, Princeton, NJ, 2010.

\bibitem[Mor11]{Mor11}
Sophie Morel.
\newblock Intersection cohomology is useless.
\newblock {\em Talk at the Institute for Advanced Study,
  https://video.ias.edu/graf/2011/03/22/morel}, 2011.

\bibitem[Nai]{Nai}
Arvind Nair.
\newblock Mixed structures in {S}himura varieties and automorphic forms.
\newblock {\em Available at http://www.math.tifr.res.in/~arvind/}.

\bibitem[Pin92]{Pin92}
Richard Pink.
\newblock On $\ell$-adic sheaves on {S}himura varieties and their higher direct
  images in the {B}aily--{B}orel compactification.
\newblock {\em Math. Ann.}, 292(2):197--240, 1992.

\bibitem[Sch15a]{Sch16}
Peter Scholze.
\newblock On torsion in the cohomology of locally symmetric varieties.
\newblock {\em Ann. of Math. (2)}, 182(3):945--1066, 2015.

\bibitem[Sch15b]{Sch15}
Peter Scholze.
\newblock On torsion in the cohomology of locally symmetric varieties.
\newblock {\em Ann. of Math. (2)}, 182(3):945--1066, 2015.

\bibitem[Ser94]{Ser94}
Jean-Pierre Serre.
\newblock Propri\'et\'es conjecturales des groupes de {G}alois motiviques et
  des repr\'esentations {$l$}-adiques.
\newblock In {\em Motives ({S}eattle, {WA}, 1991)}, volume~55 of {\em Proc.
  Sympos. Pure Math.}, pages 377--400. Amer. Math. Soc., Providence, RI, 1994.

\bibitem[Ser12]{Ser12}
Jean-Pierre Serre.
\newblock {\em Lectures on {$N_X (p)$}}, volume~11 of {\em Chapman \& Hall/CRC
  Research Notes in Mathematics}.
\newblock CRC Press, Boca Raton, FL, 2012.

\bibitem[Tat79]{Tat79}
J.~Tate.
\newblock Number theoretic background.
\newblock In {\em Automorphic forms, representations and {$L$}-functions
  ({P}roc. {S}ympos. {P}ure {M}ath., {O}regon {S}tate {U}niv., {C}orvallis,
  {O}re., 1977), {P}art 2}, Proc. Sympos. Pure Math., XXXIII, pages 3--26.
  Amer. Math. Soc., Providence, R.I., 1979.

\bibitem[Tb16]{Tai16}
Olivier Ta\"\i~bi.
\newblock Eigenvarieties for classical groups and complex conjugations in
  {G}alois representations.
\newblock {\em Math. Res. Lett.}, 23(4):1167--1220, 2016.

\bibitem[vGT94]{Gee94}
Bert van Geemen and Jaap Top.
\newblock A non-selfdual automorphic representation of {${\rm GL}_3$} and a
  {G}alois representation.
\newblock {\em Invent. Math.}, 117(3):391--401, 1994.

\bibitem[Vog84]{Vog84a}
David~A. Vogan, Jr.
\newblock Unitarizability of certain series of representations.
\newblock {\em Ann. of Math. (2)}, 120(1):141--187, 1984.

\bibitem[VZ84]{Vog84}
David~A. Vogan, Jr. and Gregg~J. Zuckerman.
\newblock Unitary representations with nonzero cohomology.
\newblock {\em Compositio Math.}, 53(1):51--90, 1984.

\end{thebibliography}

\end{document}